\documentclass{article}

\usepackage{amsmath,amssymb,amsthm}
\usepackage[a4paper]{geometry}
\usepackage{bm}
\usepackage{algorithm,algorithmic}
\usepackage{graphicx}
\usepackage{listings}
\usepackage[colorlinks=false,pdfborder={0 0 0}]{hyperref}

\graphicspath{{fig/}}

\newenvironment{keywords}{\medskip\textbf{Keywords:}}{}
\newenvironment{AMS}{\medskip\textbf{AMS subject classifications (2020).}}{}

\newtheorem{theorem}{Theorem}
\newtheorem{lemma}[theorem]{Lemma}

\DeclareMathOperator*{\argmax}{\arg\max}

\DeclareMathOperator{\rank}{rank}

\DeclareMathOperator{\sign}{sign}
\DeclareMathOperator{\Span}{span}

\newcommand{\fro}{\mathsf F}

\newcommand*{\set}[1]{\left\lbrace#1\right\rbrace}

\newcommand*{\innerp}[1]{\left\langle#1\right\rangle}

\newcommand*{\trans}{^{\top}}
\newcommand*{\herm}{^{\mathsf H}}

\newcommand*{\iherm}{^{-\mathsf H}}

\newcommand{\Eqalign}[1]{\begin{align*}#1\end{align*}}
\newcommand{\Cases}[1]{\begin{cases}#1\end{cases}}

\newcommand{\bmat}[1]{\begin{bmatrix}#1\end{bmatrix}}

\def\adots{\mathinner{\mkern2mu\raise1pt\hbox{.}\mkern2mu
    \raise4pt\hbox{.}\mkern2mu\raise7pt\hbox{.}\mkern1mu}}

\DeclareMathOperator{\fl}{f{}l}
\newcommand*{\macheps}{\bm u}
\newcommand*{\epsn}{\epsilon_{\textup{n}}}
\newcommand*{\epso}{\epsilon_{\textup{o}}}
\newcommand*{\epsB}{\epsilon_{\textup{B}}}

\begin{document}

\title{Householder orthogonalization with a non-standard inner product%
\footnote{This work is partially supported by the National Natural Science
Foundation of China under grant No.~11971118.}}
\author{Meiyue Shao%
\footnote{School of Data Science and
MOE Key Laboratory for Computational Physical Sciences,
Fudan University, Shanghai 200433, China.
Email: \texttt{myshao@fudan.edu.cn}}}

\maketitle

\begin{abstract}
Householder orthogonalization plays an important role in numerical linear
algebra.
It attains perfect orthogonality regardless of the conditioning of the input.
However, in the context of a non-standard inner product, it becomes difficult
to apply Householder orthogonalization due to the lack of an initial
orthonormal basis.
We propose strategies to overcome this obstacle and discuss algorithms and
variants of Householder orthogonalization with a non-standard inner product.
Theoretical analysis and numerical experiments demonstrate that our approach
is numerically stable under mild assumptions.


\end{abstract}

\section{Introduction}
\label{sec:introduction}

Let \(B\in\mathbb C^{n\times n}\) be Hermitian and positive definite, and
\(X\in\mathbb C^{n\times k}\) with \(n\geq k\).
There exist \(Q\in\mathbb C^{n\times k}\) and an upper triangular matrix
\(R\in\mathbb C^{k\times k}\) such that
\begin{gather}
\label{eq:thin-QR}
X=QR,\\
\label{eq:B-orth}
Q\herm BQ=I_k.
\end{gather}
The factorization~\eqref{eq:thin-QR} is known as a \emph{thin QR
factorization}~\cite{GV2013} or a \emph{reduced QR
factorization}~\cite{TB1997} of the matrix \(X\).
Condition~\eqref{eq:B-orth} means that the columns of \(Q\) form an
orthonormal basis of \(\Span(Q)\) in the non-standard inner product
\(\innerp{u,v}_B=v\herm Bu\) induced by \(B\), where the
superscript~\(\cdot\herm\) stands for the conjugate transpose of a matrix.
Throughout this paper we call \(\innerp{\cdot,\cdot}_B\) the \emph{\(B\)-inner
product}, and use QR factorization for short to refer to the thin QR
factorization~\eqref{eq:thin-QR}.

The QR factorization~\eqref{eq:thin-QR} have applications in various numerical
algorithms.
One application arises from the generalized linear least squares problem
\begin{equation}
\label{eq:GLS}
\min_x\lVert Ax-b\rVert_B,
\end{equation}
where \(\lVert v\rVert_B=\innerp{v,v}_B^{1/2}\) denotes the
\emph{\(B\)-norm}.
The solution of~\eqref{eq:GLS} can be extracted by computing the QR
factorization of \(A\)~\cite{Bjoerck1996,HV1997,Reid2000}.
Many iterative linear solvers, e.g., the conjugate gradient (CG) method and
the preconditioned (generalized) minimal residual (GMRES/MINRES) method, also
perform orthogonalization in a non-standard inner
product~\cite{GV2013,Saad2003}, albeit implicitly.
When solving Hermitian--definite generalized eigenvalue problems
\(Ax=Bx\lambda\) or product eigenvalue problems \(ABx=x\lambda\) using Krylov
subspace methods, where \(B\) is positive definite, orthogonalization with the
\(B\)-inner product is often required many times---once in each
iteration~\cite{DSYG2018,HL2006,Saad2011,VBSGY2017}.

When \(B=I_k\), there are mainly two classes of algorithms for
orthogonalization.
One class of algorithms, including Householder-QR~\cite{Householder1958},
Givens-QR~\cite{Givens1958}, and their variants~(e.g., TSQR~\cite{DGHL2012}),
performs orthogonalization by applying row transformation to the input.
We shall call these algorithms \emph{row-wise} algorithms.
Row-wise algorithms eliminate the lower triangular part of~\(X\) by applying a
sequence of unitary operators from the left and eventually transform \(X\) to
an upper triangular matrix \(R\).
The unitary operators are then assembled to form the matrix \(Q\).
The other class of algorithms performs appropriate linear combinations of the
columns of \(X\) to find an orthonormal basis of \(\Span(X)\).
We shall call this class of algorithms \emph{column-wise} algorithms.
Classical Gram--Schmidt (CGS) and modified Gram--Schmidt (MGS)
processes~\cite{Bjoerck1994,LBG2013} are typical column-wise algorithms.
Another frequently used column-wise algorithm is the Cholesky-QR
algorithm~\cite{FKNYY2020,YNYF2016}, which first computes the Cholesky factor
\(R\) of the positive definite matrix \(X\herm X\) and then solve a triangular
system to obtain \(Q=XR^{-1}\).
When only an orthonormal basis of \(\Span(X)\) is of interest, the Cholesky
factorization of \(X\herm X\) can be replaced by the spectral decomposition
(after proper scaling), leading to the SVQB algorithm~\cite{SW2002}.

Column-wise algorithms have the advantage that they naturally carry over to
the case of non-standard inner products, and, in addition, both Cholesky-QR
and SVQB are suitable for high performance computing because they have level~3
arithmetic intensity.
However, these algorithms often suffer from numerical instability, and usually
require reorthogonalization to improve the
orthogonality~\cite{LBG2013,YNYF2016}.

Compared to column-wise algorithms, row-wise algorithms have much better
numerical stability even if \(X\) is rank deficient, because the orthonormal
basis is extracted from columns of a (computed) unitary matrix.
However, it becomes non-trivial to extend row-wise algorithms to work with a
non-standard inner product.
We shall see that in the context of a non-standard inner product, a matrix
representing a unitary operator usually does \emph{not} have orthonormal
columns.

Recently, Trefethen~\cite{Trefethen2010} generalized the method of Householder
orthogonalization to the infinite dimensional Hilbert space \(L^2[a,b]\).
The basic idea is to map a set of vectors to a prescribed orthonormal basis,
instead of directly eliminating matrix entries of the input.
As long as an orthonormal basis is available, Trefethen's algorithm can be
generalized to work with any separable Hilbert space.
Thus it is possible to apply Trefethen's algorithm to
compute the QR factorization~\eqref{eq:thin-QR} with the \(B\)-inner product.
However, by far Trefethen's algorithm is not considered practical for
orthogonalization with the \(B\)-inner product unless an existing orthonormal
basis in the \(B\)-inner product is already available~\cite{LL2014,IY2019}.
We shall discuss how to tackle this issue in practice under mild assumptions.

In this paper we study the computation of~\eqref{eq:thin-QR} based on
Householder reflections.
We assume that \(n\gg k\), and \(B\) is not too ill-conditioned so that
the \(B\)-inner product \(\innerp{v,w}_B\) can be evaluated reasonably
accurately.
These assumptions are valid in many Krylov subspace eigensolvers.
We do \emph{not} require \(X\) to have full column rank---\(\Span(X)\) becomes
a proper subspace of \(\Span(Q)\) if \(\rank(X)<k\).

The rest of this paper is organized as follows.
In Section~\ref{sec:preliminary}, we first recall some basic properties of
Householder reflections.
In Section~\ref{sec:algorithm} we present algorithms and variants for
Householder orthogonalization with the \(B\)-inner product.
Strategies for constructing an initial orthonormal basis are discussed
in Section~\ref{sec:initial}.
A brief stability analysis is provided in Section~\ref{sec:rounding}.
We demonstrate by numerical examples the effectiveness of the proposed
algorithms in Section~\ref{sec:experiments}.
Finally, the paper is concluded in Section~\ref{sec:conclusions}.

\section{Householder reflections}
\label{sec:preliminary}
In this section we briefly recall some basic properties of Householder
reflections.

Let \(\mathcal V\) be a vector space over the field \(\mathbb F\)
(for \(\mathbb F\in\set{\mathbb R,\mathbb C}\)) equipped with an inner product
\(\innerp{\cdot,\cdot}\).
The norm of a vector \(x\in\mathcal V\) is given by
\(\lVert x\rVert=\innerp{x,x}^{1/2}\).
A set of vectors \(\set{x_1,x_2,\dotsc}\subset\mathcal V\) is called
\emph{orthonormal} if
\[
\innerp{x_i,x_j}=\Cases{1 & \text{(\(i=j\))}, \\ 0 & \text{(\(i\neq j\))}.}
\]

Let \(w\in\mathcal V\) be a unit vector, i.e., \(\lVert w\rVert=1\).
A linear operator \(H\colon\mathcal V\to\mathcal V\) defined by
\[
Hx=x-2w\innerp{w,x} \qquad (\forall x\in\mathcal V)
\]
is called a \emph{Householder reflection} or, more precisely, a Householder
reflection with respect to the hyperplane orthogonal to \(w\).
The vector \(w\) is known as the \emph{Householder vector}.
It can be easily verified by definition that
\(H\) is \emph{self-adjoint} (i.e., \(\innerp{Hx,y}=\innerp{x,Hy}\) for
\(x\), \(y\in\mathcal V\)), \emph{unitary} (i.e.,
\(\innerp{Hx,Hy}=\innerp{x,y}\) for \(x\), \(y\in\mathcal V\)), and
\emph{involutory} (i.e., \(H^2x=x\) for \(x\in\mathcal V\)).
In addition, for any \(x\), \(y\in\mathcal V\) satisfying
\(\lVert x\rVert=\lVert y\rVert\) and \(\innerp{x,y}\in\mathbb R\), we can
construct a self-adjoint and unitary linear operator of the form
\(H(\cdot)=I-2w\innerp{w,\cdot}\) by choosing
\[
\Span\set{w}=\Span\set{x-y}, \qquad (\text{\(\|w\|=1\) or \(\|w\|=0\)}),%
\footnote{We allow \(w=0\) in the case \(x=y\) to simplify the notation.}
\]
such that \(Hx=y\) and \(Hy=x\).
This property plays a key role in the construction of Householder reflections.

When \(\mathcal V=\mathbb F^n\) and the inner product is the \(B\)-inner
product induced by a positive definite matrix \(B\in\mathbb F^{n\times n}\),
the situation is slightly more complicated compared to that in an abstract
setting, since the canonical basis (i.e., columns of the identity matrix
\(I_n\)) is in general \emph{not} orthonormal in the \(B\)-inner product.
It is worth noting that there are two different types of orthogonality induced
by \(B\).
The condition \(U\herm BU=I_k\) for \(U\in\mathbb F^{n\times k}\) means that
the columns of \(U\) form an orthonormal basis, while the condition \(V\herm
BV=B\) for \(V\in\mathbb F^{n\times n}\) implies that \(V\) is the matrix
representation of a unitary operator.
These two types of orthogonality are related through
\begin{equation}
\label{eq:orth-orth}
(VU)\herm B(VU)=U\herm(V\herm BV)U=U\herm BU=I_k,
\end{equation}
which reveals the fact that the unitary operator \(V\) maps one orthonormal
set \(U\) to another orthonormal set \(VU\).

A Householder reflection with the \(B\)-inner product is of the form
\[
H=I_n-2ww\herm B,
\]
where the Householder vector \(w\) satisfies \(\lVert w\rVert_B=1\).
Note that \(H\herm=H\) does not hold in general when \(B\neq I_n\).
But a number of properties of Householder reflections in the standard inner
product remain valid.
For instance, we have \(H^2=I_n\), \(H^{-1}=H\), \(H\herm BH=B\),
\(\lVert Hx\rVert_B=\lVert x\Vert_B\), \(\det(H)=-1\), and so on.
When \(\lVert x\rVert_B=\lVert y\rVert_B\neq0\), we can find a Householder
reflection \(H\) such that \(y=Hx=H^{-1}x\).
In the subsequent section we shall see that the basic idea of Householder
orthogonalization is actually using Householder reflections to map a
prescribed orthonormal set to the desired one.

\section{Householder orthogonalization}
\label{sec:algorithm}
In the following we discuss how to make use of Householder reflections to
compute the QR factorization~\eqref{eq:thin-QR} in the \(B\)-inner product.
Throughout this section we assume that an orthonormal set
\(\set{u_1,u_2,\dotsc,u_k}\) is available, i.e., we already have a matrix
\(U=[u_1,u_2,\dotsc,u_k]\in\mathbb F^{n\times k}\) such that
\(U\herm BU=I_k\).
Strategies for constructing such a matrix \(U\) will be provided in
Section~\ref{sec:initial}.

\subsection{Right-looking algorithm}
\label{subsec:right-looking}
We first discuss a right-looking algorithm for computing the QR factorization.
To see how to apply Householder reflections in orthogonalization, we partition
\(X\) and \(Q\) into columns, and rewrite~\eqref{eq:thin-QR} as
\begin{equation}
\label{eq:thin-QR2}
[x_1,x_2,\dotsc,x_k]=[q_1,q_2,\dotsc,q_k]
\bmat{r_{1,1} & r_{1,2} & \cdots & r_{1,k} \\
& r_{2,2} & \cdots & r_{2,k} \\
& & \ddots & \vdots \\
& & & r_{k,k}}.
\end{equation}

Let us use the superscript \(\cdot^{(i)}\) to represent the quantity that
is overwritten at the \(i\)th step of the algorithm.
Initially, we have \(X=X^{(0)}\).
Let \(r_{1,1}=\lVert x_1^{(0)}\rVert_B\), and \(v_1=x_1^{(0)}/r_{1,1}\) if
\(r_{1,1}>0\).
A Householder reflection \(H_1=I_n-2w_1w_1\herm B\) is chosen such that
\(H_1u_1=v_1\), where
\[
w_1=\frac{v_1-u_1\alpha_1}{\lVert v_1-u_1\alpha_1\rVert_B},
\qquad
\alpha_1=\argmax_{\alpha\in\mathbb F,~\lvert\alpha\rvert=1}
\lVert v_1-u_1\alpha\rVert_B=-\sign(u_1\herm Bv_1).
\]
By choosing \(\alpha_1\) in such a way, cancellation can be avoided when
computing the Householder vector \(w_1\).
In case that \(r_{1,1}=0\), we simply set \(w_1=0\) and there is no need to
form \(v_1\) and \(\theta_1\).

The next step is to remove components contributed by \(u_1\) from
\(H_1x_2^{(0)}\), \(\dotsc\), \(H_1x_k^{(0)}\).
This is accomplished by setting
\[
r_{1,i}=u_1\herm B(H_1x_i^{(0)}), \qquad x_i^{(1)}=H_1x_i^{(0)}-u_1r_{1,i},
\]
for \(i=2\), \(\dotsc\), \(k\).
So far we have arrived at
\begin{align}
H_1\bigl[x_1^{(0)},x_2^{(0)},\dotsc,x_k^{(0)}\bigr]
&=\bigl[u_1r_{1,1},H_1x_2^{(0)},\dotsc,H_1x_k^{(0)}\bigr]\nonumber\\
&=\bigl[u_1,x_2^{(1)},\dotsc,x_k^{(1)}\bigr]
\bmat{r_{1,1} & r_{1,2} & \cdots & r_{1,k} \\
& 1 & & \\
& & \ddots & \\
& & & 1}.
\label{eq:rl-1}
\end{align}

A notable property is that \(u_1\) is perpendicular to
\(\lbrace x_2^{(1)},\dotsc,x_k^{(1)},u_2,\dotsc,u_k\rbrace\) in the
\(B\)-inner product.
By applying the same procedure recursively to
\(\lbrace x_2^{(1)},\dotsc,x_k^{(1)}\rbrace\) in the \((n-1)\)-dimensional
subspace \(\Span\set{u_1}^{\perp_B}\),%
\footnote{The notation \(\cdot^{\perp_B}\) denotes the orthogonal complement
with the \(B\)-inner product.}
we obtain Householder reflections of the form \(H_i=I_n-2w_iw_i\herm B\) such
that
\begin{equation}
H_k\dotsm H_3H_2[x_2^{(1)},x_3^{(1)},\dotsc,x_k^{(1)}]
=[u_2,u_3,\dotsc,u_k]
\bmat{r_{2,2} & r_{2,3} & \cdots & r_{2,k} \\
& r_{3,3} & \cdots & r_{3,k}\\
& & \ddots & \vdots \\
& & & r_{k,k}}.
\label{eq:rl-2}
\end{equation}
When \(1\leq j<i\leq k\), it follows from \(u_j\herm Bw_i=0\) that
\begin{equation}
\label{eq:invariant}
H_iu_j=u_j, \qquad (i>j).
\end{equation}
Therefore, we conclude that
\[
H_k\dotsm H_2H_1[x_1,x_2,\dotsc,x_k]
=[u_1,u_2,\dotsc,u_k]
\bmat{r_{1,1} & r_{1,2} & \cdots & r_{1,k} \\
& r_{2,2} & \cdots & r_{2,k}\\
& & \ddots & \vdots \\
& & & r_{k,k}}
\]
by combining~\eqref{eq:rl-1} and~\eqref{eq:rl-2}.
Let
\begin{equation}
\label{eq:full}
[q_1,q_2,\dotsc,q_k]=H_1H_2\dotsm H_k[u_1,u_2,\dotsc,u_k],
\end{equation}
or, mathematically equivalently,
\begin{equation}
\label{eq:half}
q_1=H_1u_1,\qquad q_2=H_1H_2u_2,\qquad\dotsc,\qquad q_k=H_1H_2\dotsm H_ku_k,
\end{equation}
according to~\eqref{eq:invariant}.
It can be easily verified using~\eqref{eq:orth-orth} that \(Q\herm BQ=I_k\).
Thus the QR factorization~\eqref{eq:thin-QR2} has been calculated.

We summarize the procedure as Algorithm~\ref{alg:right-looking}.
It is essentially Trefethen's algorithm, reformulated in a way that we can
clearly see from
steps~\ref{alg-step:right-looking0}--\ref{alg-step:right-looking1} that this
algorithm is in fact a right-looking algorithm.
Step~\ref{alg-step:reorth-right} in Algorithm~\ref{alg:right-looking} is a
reorthogonalization step that mathematically does not change \(w_i\)'s.
However, just like the case for \(L^2[a,b]\) as discussed
in~\cite{Trefethen2010}, it is \emph{not} recommended to skip this step in
practice due to the presence of rounding errors, unless certain special
structures of \(B\) can be exploited to ensure numerical stability.%
\footnote{For instance, when \(B\) is diagonal, \(w_i\) is guaranteed to be
orthogonal to \(\Span\set{u_1,u_2,\dotsc,u_{i-1}}\) both theoretically and
numerically.
In this case reorthogonalization becomes unnecessary.}
We shall see in Section~\ref{sec:experiments} that skipping such a
reorthogonalization step can cause large residual \(\lVert X-QR\rVert_2\) when
\(X\) is ill-conditioned.
In practice, we observe that the classical Gram--Schmidt process (without
further reorthogonalization) is already sufficiently accurate for this step.
Finally, we remark that column pivoting can be incorporated in
Algorithm~\ref{alg:right-looking} so that the diagonal entries of \(R\) are
decreasing.

\begin{algorithm}[!tb]
\caption{Right-looking Householder orthogonalization algorithm}
\label{alg:right-looking}
\begin{algorithmic}[1]
\REQUIRE , A matrix \(X\in\mathbb F^{n\times k}\) to be orthogonalized, a
positive definite matrix \(B\in\mathbb F^{n\times n}\),
and \(U\in\mathbb F^{n\times k}\) such that \(U\herm BU=I_k\).
\ENSURE Matrices \(Q\in\mathbb F^{n\times k}\) and
\(R\in\mathbb F^{k\times k}\) satisfying~\eqref{eq:thin-QR}
and~\eqref{eq:B-orth}.
\FOR{\(i=1\) \TO \(k\)}
  \STATE \(r_{i,i}\gets\lVert x_i\rVert_B\).
  \IF{\(r_{i,i}=0\)}
    \STATE \(w_i\gets0\).
  \ELSE
    \STATE \(x_i\gets x_i/r_{i,i}\).
    \STATE \(u_i\gets-\sign(u_i\herm Bx_i)u_i\).
    \STATE \(w_i\gets x_i-u_i\).
    \STATE (optional) \(w_i\gets w_i-[u_1,\dotsc,u_{i-1}]
      [u_1,\dotsc,u_{i-1}]\herm Bw_i\).
    \label{alg-step:reorth-right}
    \STATE \(w_i\gets w_i/\lVert w_i\rVert_B\).
    \STATE \([x_{i+1},\dotsc,x_k]\gets[x_{i+1},\dotsc,x_k]
      -2w_iw_i\herm B[x_{i+1},\dotsc,x_k]\).
    \label{alg-step:right-looking0}
    \STATE \([r_{i,i+1}\dotsc,r_{i,k}]\gets u_i\herm B[x_{i+1},\dotsc,x_k]\).
    \STATE \([x_{i+1},\dotsc,x_k]\gets[x_{i+1},\dotsc,x_k]
      -u_i[r_{i,i+1},\dotsc,r_{i,k}]\).
    \label{alg-step:right-looking1}
  \ENDIF
\ENDFOR
\STATE \([q_1,\dotsc,q_k]\gets[u_1,\dotsc,u_k]\).
\FOR{\(i=k\) \TO \(1\)}
  \STATE \([q_i,\dotsc,q_k]\gets[q_i,\dotsc,q_k]
    -2w_iw_i\herm B[q_i,\dotsc,q_n]\).
\ENDFOR
\end{algorithmic}
\end{algorithm}

\subsection{Left-looking algorithm}
\label{subsec:left-looking}
In the following we present a left-looking version of Householder
orthogonalization.
Albeit being mathematically equivalent to the right-looking algorithm, the
left-looking version is suitable in the context of (block) Lanczos/Arnoldi
process in which \(x_1\), \(\dotsc\), \(x_k\) are gradually produced on the
fly instead of being immediately available at the beginning.

Let us assume that we have already obtained the QR factorization of
\([x_1,\dotsc,x_{i-1}]\) as
\[
[x_1,\dotsc,x_{i-1}]=H_1\dotsm H_{i-1}[u_1,\dotsc,u_{i-1}]R^{(i-1)},
\]
where \(R^{(i-1)}\in\mathbb F^{(i-1)\times(i-1)}\).
Applying one step of classical Gram--Schmidt process on the vector
\(H_{i-1}\dotsm H_1x_i\) yields
\[
H_{i-1}\dotsm H_1x_i=[u_1,\dotsc,u_{i-1}]r_i+v_i\beta_i,
\]
where $v_i\in\Span\set{u_1,\dotsc,u_{i-1}}^{\perp_B}$ is a unit vector, i.e.,
\(\lVert v_i\rVert_B=1\).
Then we construct a Householder reflection \(H_i\) satisfying \(H_iu_i=v_i\)
and \(H_iu_j=u_j\) for \(1\leq j<i\).%
\footnote{Similar to the case in the right-looking algorithm, we simply choose
\(H_i=I_n\) if \(\beta_i=0\).}
By choosing
\[
R^{(i)}=\bmat{R^{(i-1)} & r_i \\ & \beta_i}\in\mathbb F^{i\times i},
\]
we deduce
\begin{align*}
[x_1,\dotsc,x_{i-1},x_i]
&=H_1\dotsm H_{i-1}[u_1,\dotsc,u_{i-1},H_{i-1}\dotsm H_1x_i]
\bmat{R^{(i-1)} & 0 \\ & 1}\\
&=H_1\dotsm H_{i-1}[u_1,\dotsc,u_{i-1},v_i]
\bmat{R^{(i-1)} & r_i \\ & \beta_i}\\
&=H_1\dotsm H_{i-1}H_i[u_1,\dotsc,u_{i-1},u_i]R^{(i)},
\end{align*}
which is the QR factorization of \([x_1,\dotsc,x_i]\).
Repeating this procedure for \(i=1\), \(2\), \(\dotsc\), \(k\) yields a
left-looking algorithm of Householder orthogonalization, as summarized in
Algorithm~\ref{alg:left-looking}.
Similar to the right-looking algorithm, a reorthogonalization step using CGS
(Step~\ref{alg-step:reorth-left}) is strongly recommended in finite precision
arithmetic.

\begin{algorithm}[!tb]
\caption{Left-looking Householder orthogonalization algorithm}
\label{alg:left-looking}
\begin{algorithmic}[1]
\REQUIRE , A matrix \(X\in\mathbb F^{n\times k}\) to be orthogonalized, a
positive definite matrix \(B\in\mathbb F^{n\times n}\),
and \(U\in\mathbb F^{n\times k}\) such that \(U\herm BU=I_k\).
\ENSURE Matrices \(Q\in\mathbb F^{n\times k}\) and
\(R\in\mathbb F^{k\times k}\) satisfying~\eqref{eq:thin-QR}
and~\eqref{eq:B-orth}.
\FOR{\(i=1\) \TO \(k\)}
  \FOR{\(j=1\) \TO \(i-1\)}
  \label{alg-step:left-looking0}
    \STATE \(x_i\gets x_i-2w_jw_j\herm Bx_i\).
  \ENDFOR
  \label{alg-step:left-looking1}
  \STATE \([r_{1,i},\dotsc,r_{i-1,i}]\trans\gets
    [u_1,\dotsc,u_{i-1}]\herm Bx_i\).
  \STATE \(x_i\gets x_i-[u_1,\dotsc,u_{i-1}]
    [r_{1,i},\dotsc,r_{i-1,i}]\trans\).
  \label{alg-step:left-looking2}
  \STATE \(r_{i,i}\gets\lVert x_i\rVert_B\).
  \IF{\(r_{i,i}=0\)}
    \STATE \(w_i\gets0\).
  \ELSE
    \STATE \(x_i\gets x_i/r_{i,i}\).
    \STATE \(u_i\gets-\sign(u_i\herm Bx_i)u_i\).
    \STATE \(w_i\gets x_i-u_i\).
    \STATE (optional) \(w_i\gets w_i-[u_1,\dotsc,u_{i-1}]
      [u_1,\dotsc,u_{i-1}]\herm Bw_i\).
    \label{alg-step:reorth-left}
    \STATE \(w_i\gets w_i/\lVert w_i\rVert_B\).
  \ENDIF
\ENDFOR
\STATE \([q_1,\dotsc,q_k]\gets[u_1,\dotsc,u_k]\).
\FOR{\(i=k\) \TO \(1\)}
  \STATE \([q_i,\dotsc,q_k]\gets[q_i,\dotsc,q_k]
    -2w_iw_i\herm B[q_i,\dotsc,q_k]\).
\ENDFOR
\end{algorithmic}
\end{algorithm}

\subsection{Compact representation of reflections}
\label{subsec:compact}
In the context of a standard inner product, i.e., \(B=I_n\), a product of a
few Householder reflections admits a compact WY representation of the form
\begin{equation}
\label{eq:WY1}
(I-2w_1w_1\herm)\dotsm(I-2w_kw_k\herm)=I_n+WTW\herm,
\end{equation}
where \(W\in\mathbb F^{n\times k}\) is lower trapezoidal and
\(T\in\mathbb F^{k\times k}\) is upper triangular~\cite{SV1989}.
The compact WY representation is important for high performance computing
since it provides opportunities for efficient applications of a sequence of
Householder reflections using level~3 BLAS.
By relaxing the requirement on the nonzero pattern of \(W\), the compact WY
representation easily carries over to Householder reflections with a
non-standard inner product.

Let \(W_i=[w_1,w_2,\dotsc,w_i]\) for \(i\leq k\), and
\(H_i=I_n-2w_iw_i\herm B\) be Householder reflections (with
\(W_k\herm BW_k=I_k\)).
Suppose that we already have a compact WY representation of the form
\[
H_1\dotsm H_{i-1}=I_n-2W_{i-1}T_{i-1}W_{i-1}\herm B,
\]
where \(T_{i-1}\in\mathbb F^{(i-1)\times(i-1)}\) is unit upper triangular.
Then it can be verified that
\begin{equation}
\label{eq:WY2}
H_1\dotsm H_{i-1}H_i=I_n-2W_iT_iW_i\herm B,
\end{equation}
where
\[
T_i=\bmat{T_{i-1} & 2T_{i-1}W_{i-1}\herm Bw_i \\ & 1}\in\mathbb F^{i\times i}
\]
is also unit upper triangular.

The compact WY representation~\eqref{eq:WY2} generalizes~\eqref{eq:WY1},
despite that a slightly different scaling convention is adopted here.
Some care needs to be taken when we need to apply the transformations in
reversed order.
Since the matrix \(H_1\dotsm H_k\) in~\eqref{eq:WY2} is not unitary (in the
standard inner product), we cannot invert the right-hand side
of~\eqref{eq:WY2}  by directly taking the
conjugate transpose.
Instead, we have
\begin{equation}
\label{eq:WY3}
H_k\dotsm H_1=(H_1\dotsm H_k)^{-1}=I_n-2W_kT_k\herm W_k\herm B,
\end{equation}
which is essentially the adjoint operator of \(H_1\dotsm H_k\) in the
\(B\)-inner product.

The compact WY representation of Householder reflections is useful in
Householder orthogonalization, especially when \(k\) is not so small.
Even for a small \(k\) we can still make use of the compact WY representation
in the left-looking algorithm.
For instance, \eqref{eq:WY3} can be used to compute
\(H_{i-1}\dotsm H_1x_i\) in Algorithm~\ref{alg:left-looking}
(steps~\ref{alg-step:left-looking0}--\ref{alg-step:left-looking1}).
When \(k\) is of medium size (e.g., \(k=O(1000)\), which is typical in
electronic structure calculations~\cite{DSYG2018}), we may design a block
algorithm that mixes Algorithms~\ref{alg:right-looking}
and~\ref{alg:left-looking}:
Partition \(X\) into a few \(n\times b\) panels, where \(1\ll b\ll k\);
use the left-looking algorithm for panel factorization, and then apply a block
version of the right-looking algorithm to update the remaining columns of
\(X\).
In such a block algorithm, arithmetic intensity can be improved by making use
of the WY representation in the right-looking update.

We remark that, if implemented carefully, the compact WY representation can
also be adopted in the final steps of Algorithms~\ref{alg:right-looking}
and~\ref{alg:left-looking}, i.e., computing
\[
[q_1,q_2,\dotsc,q_k]=H_1H_2\dotsm H_k[u_1,u_2,\dotsc,u_k].
\]
A naive application of~\eqref{eq:WY2} without taking into
account~\eqref{eq:invariant} roughly doubles the computational cost, and
increases rounding errors.
Hence a block approach is more appropriate when \(k\) of modest size.
However, detailed discussions concerning high performance computing aspects is
beyond the scope of this paper.

\section{Construction of an initial orthonormal basis}
\label{sec:initial}
In the previous section, we assume that a matrix \(U\in\mathbb F^{n\times k}\)
satisfying \(U\herm BU=I_k\) is already available.
The algorithms heavily relies on the availability of \(U\).
In Trefethen's original work~\cite{Trefethen2010}, the set of scaled Legendre
polynomials is naturally available as an orthonormal basis of \(L^2[a,b]\).
However, in general an obvious choice does not exist for \(F^n\) with the
\(B\)-inner product.
In this section we discuss how to construct an initial orthonormal basis in
practice.

When \(k\) is close to \(n\), finding a matrix \(U\) with \(U\herm BU=I_k\) is
not much cheaper than computing the Cholesky factorization of \(B\) unless
certain special structures of \(B\) can be exploited.
This is the main obstacle for Householder orthogonalization in a very generic
setting.
However, we are mainly interested in the case \(k\ll n\), which is the typical
situation in practice.
In this case there is no need to find an entire orthonormal basis of
\(\mathbb F^n\).
Instead, an orthonormal basis of \emph{any} \(k\)-dimensional subspace
suffices.

Let \(\tilde B\) be the leading \(k\times k\) principal submatrix of \(B\).
Then \(\tilde B\) is positive definite, and admits a Cholesky factorization
\(\tilde B=\tilde R\herm\tilde R\).
It can be easily verified that
\[
U=\bmat{\tilde R^{-1} \\ 0}\in\mathbb F^{n\times k}
\]
satisfies \(U\herm BU=\tilde R\iherm\tilde B\tilde R^{-1}=I_k\).
If the matrix \(B\) is explicitly stored, constructing \(U\) requires as cheap
as \(O(k^3)\) operations and \(O(k^2)\) storage.
This cost is negligible, since any practical orthogonalization algorithm
requires \(O(nk^2)\) dense linear algebra operations in addition to
\(O(k)\) matrix--vector multiplications.
Even if the matrix \(B\) is only implicitly available through a black box
function \(x\mapsto Bx\), we have \(\tilde B=E_k\herm BE_k\), where
\(E_k=[e_1,e_2,\dotsc,e_k]\) denotes the leading columns of the identity
matrix \(I_n\).
The cost consists of \(O(k^3)\) dense operations and \(O(k)\) matrix--vector
multiplications, which is still lower than the entire orthogonalization
algorithm.
Therefore, we have obtained a strategy, as summarized in
Algorithm~\ref{alg:init}, to construct \(U\) with affordable overhead.

\begin{algorithm}[!tb]
\caption{Construction of an initial orthonormal set}
\label{alg:init}
\begin{algorithmic}[1]
\REQUIRE A positive definite matrix \(B\in\mathbb F^{n\times n}\), a positive
integer \(k\leq n\).
\ENSURE A matrix \(U\in\mathbb F^{n\times k}\) such that \(U\herm BU=I_k\).
\STATE \(\tilde B\gets\) the \(k\times k\) leading submatrix of \(B\).
\STATE Compute the Cholesky factorization \(\tilde B=\tilde R\herm\tilde R\).
\STATE \(U\gets\bmat{\tilde R\iherm,0}\herm\).
\end{algorithmic}
\end{algorithm}

Algorithm~\ref{alg:init} can be viewed as the Cholesky-QR algorithm applied to
\(E_k\).
The idea is similar to that of the PRECHOL-QR algorithm in~\cite{LL2014}.
Unlike the general Cholesky-QR algorithm which is very sensitive to rounding
errors, we use a well-conditioned input with \(\kappa_2(E_k)=1\),%
\footnote{The notation
\(\kappa_2(A)=\lVert A\rVert_2\lVert A^{\dagger}\rVert_2\)
denotes the condition number of \(A\), where \(A^{\dagger}\) is the
Moore--Penrose pseudoinverse of \(A\).}
and the Gramian matrix \(\tilde B\) is often reasonably well-conditioned.
Theoretically, it follows from the Cauchy interlacing theorem that
\(\kappa_2(\tilde B)\leq\kappa_2(B)\).
Though the extreme case \(\kappa_2(\tilde B)=\kappa_2(B)\) may occur, very
often we can even expect \(\kappa_2(\tilde B)\ll\kappa_2(B)\) in practice as
long as \(k\) is not too large.

Certainly there exist many alternatives to Algorithm~\ref{alg:init}.
For instance, choosing any columns from the identity matrix works equally well.
Besides the canonical basis with the standard inner product, other
well-conditioned basis can also be used to construct \(U\), as long as the
Gramian matrix is not too ill-conditioned.
The freedom of replacing \(E_k\) by other matrices comes at a price of
possibly higher rounding errors.
In the case that \(\tilde B\) is ill-conditioned, some iterative refinement
strategies (e.g., shifted Cholesky-QR~\cite{FKNYY2020}, possibly using
extended precision arithmetic~\cite{YTD2015}) can be adopted to improve the
orthogonality of \(U\).

In the context of Krylov subspace methods that perform orthogonalization in
each iteration, there are alternative strategies for constructing \(U\).
For instance, we may use another orthogonalization algorithm (e.g., modified
Gram--Schmidt process with reorthogonalization) in the first iteration to
obtain an orthonormal basis, and then use this basis as \(U\) in the
subsequent iterations.
Algorithm~\ref{alg:init} also becomes more appealing compared to the case that
only one set of vectors needs to be orthogonalized, since \(U\) is constructed
only once and then can be reused many times.
Hence, the overhead for constructing \(U\) becomes negligible in this setting,
even by taking into account the cost of optional iterative refinement with
extended precision arithmetic.

\section{Numerical orthogonality}
\label{sec:rounding}
In the following we discuss the numerical orthogonality of Householder
reflections with the \(B\)-inner product in finite precision arithmetic.
We shall show that under mild assumptions the computed Householder reflections
are orthogonal.

\subsection{Rounding models}
On most modern computational units, if there is no overflow or (gradual)
underflow in the calculation, we can assume that the computed results for
finite numbers, denoted by \(\fl(\cdot)\)'s, satisfy
\begin{equation}
\label{eq:rounding}
\fl(a\circ b)=(a\circ b)(1+\epsilon_1),
\qquad
\fl(c^{1/2})=c^{1/2}(1+\epsilon_2),
\end{equation}
for \(a\), \(b\in\mathbb F\), \(\circ\in\set{+,-,\times,/}\),
\((b,\circ)\neq(0,/)\), and \(c\in[0,+\infty)\), where the \emph{machine
precision} \(\macheps\) is an upper bound of
\(\max\set{\lvert\epsilon_1\rvert,\lvert\epsilon_2\rvert}\).
The rounding model~\eqref{eq:rounding} is valid for both real and complex
arithmetic since \(\macheps_{\mathrm{complex}}\) is a small multiple of
\(\macheps_{\mathrm{real}}\)~\cite[Section~3.6]{Higham2002}.

Unlike existing rounding error analysis directly based on~\eqref{eq:rounding}
as in~\cite{LL2014,RTSK2012,YNYF2016}, we provide a higher level abstraction
that illustrates major sources of rounding errors.
We first assume that the \(B\)-inner product is evaluated in a backward stable
manner in the sense that
\begin{equation}
\label{eq:innerprod}
\fl(v\herm Bw)=v\herm(B+\Delta B)w,
\qquad \lVert\Delta B\rVert_2\leq\epsB\lVert B\rVert_2,
\end{equation}
where \(\epsB\geq\macheps\) is a constant.
In fact, if \(v\neq0\) and \(w\neq0\), a rank-one backward error can be chosen
as
\[
\Delta B=\frac{\fl(v\herm Bw)-v\herm Bw}
{\lVert v\rVert_2\lVert w\rVert_2}
\cdot\frac{vw\herm}{\lVert v\rVert_2\lVert w\rVert_2}.
\]
The model~\eqref{eq:innerprod} is plausible because it merely requires a weak
assumption that
\[
\sup_{vw\herm\neq0}\frac{\bigl\lvert\fl(v\herm Bw)-v\herm Bw\bigr\rvert}
{\lVert B\rVert_2\lVert v\rVert_2\lVert w\rVert_2}<+\infty.
\]

Based on the model~\eqref{eq:innerprod}, we conclude that evaluating the \(B\)
-norm of a nonzero vector \(w\) introduces a relative error bounded by
\(\kappa_2(B)\epsB\) because
\[
\frac{\bigl\lvert\fl(w\herm Bw)-w\herm Bw\bigr\rvert}{w\herm Bw}
=\frac{\bigl\lvert w\herm\Delta Bw\bigr\rvert}{w\herm Bw}
\leq\lVert B^{-1/2}\Delta BB^{-1/2}\rVert_2
\leq\kappa_2(B)\epsB.
\]
Therefore, we make another assumption that normalizing a vector with the
\(B\)-norm in finite precision arithmetic introduces an error no more than
\(\epsn\).
To be more precise, the normalization operation
\(\hat w\gets w/(w\herm Bw)^{1/2}\) produces an approximate unit vector
\(\hat w\) that satisfies
\begin{equation}
\label{eq:normalization}
\lVert\hat w\rVert_B^2=\hat w\herm B\hat w=1+\epsilon,
\qquad \lvert\epsilon\rvert\leq\epsn,
\end{equation}
where \(\epsn\) represents the \emph{normalization error bound}.
When \(B\) is numerically positive definite, we can assume
\(\kappa_2(B)\epsB<1\).
Then there exists
\[
\epsn\leq\frac{1}{1-\kappa_2(B)\epsB}-1
=\frac{\kappa_2(B)\epsB}{1-\kappa_2(B)\epsB}
=O\bigl(\kappa_2(B)\epsB\bigr).
\]

Though \(\epsB\) and \(\epsn\) are closely related, instead of bounding all
errors in terms of \(\epsB\), we shall make use of \(\epsn\) whenever
possible.
Normalization errors are ubiquitous in all orthogonalization algorithms
(including Gram--Schmidt processes, Cholesky-QR, SVQB, as well as
Householder-QR) that involve the evaluation of \(B\)-inner products.
In the extreme case, even normalizing a perfectly orthogonal set of vectors
introduces normalization errors.
A necessary condition for any orthogonalization algorithm to proceed without
breakdown is \(\epsn<1\), i.e., \(B\) at least needs to be numerically
positive definite.
It can be expected that \(\epsn\) plays an important role in characterizing
the numerical orthogonality.

While in general we expect \(\epsB=O(\macheps)\) and
\(\epsn=O\bigl(\kappa_2(B)\macheps\bigr)\), the concrete values of \(\epsB\)
and~\(\epsn\) depend on how accurately the matrix--vector multiplication
\(w\mapsto Bw\) is performed.
Using \(\epsB\) and \(\epsn\) treats the operation \(w\mapsto Bw\) as a black
box, while still allows algorithmic details of matrix--vector multiplication
to be taken into account.
For instance, if the entries of \(B\) are explicitly available and
\(2n\macheps<1\), it can be shown~(see, e.g., \cite[Section~3]{Higham2002})
that
\[
\lvert\Delta B\rvert\leq\frac{2n\macheps}{1-2n\macheps}\lvert B\rvert,
\footnote{Both the notation of absolute value and the inequality are
understood entrywise.}
\]
and hence
\[
\frac{\lVert\Delta B\rVert_2}{\lVert B\rVert_2}
\leq\frac{2n\macheps}{1-2n\macheps}\cdot
\frac{\bigl\lVert\lvert B\rvert\bigr\rVert_2}{\lVert B\rVert_2}
\leq\frac{2n\macheps}{1-2n\macheps}\cdot
\frac{\bigl\lVert\lvert B\rvert\bigr\rVert_{\fro}}
{n^{-1/2}\lVert B\rVert_{\fro}}
=\frac{2n^{3/2}\macheps}{1-2n\macheps}.
\]
It is then possible to substitute \(\epsB\) by the corresponding upper bound.
Alternative estimates of~\(\epsB\) and \(\epsn\), whenever available, are also
applicable.

\subsection{Orthogonality of computed Householder reflections}
A computed Householder reflection \(\hat H_i=I_n-2\hat w_i\hat w_i\herm B\) is
represented by the corresponding Householder vector \(\hat w_i\).
By our assumption~\eqref{eq:normalization}, \(\hat w_i\) satisfies
\(\lvert\hat w_i\herm B\hat w_i-1\rvert\leq\epsn\).
We first establish a structured backward error estimate regarding the
orthogonality of Householder reflections as shown in Lemma~\ref{lem:perturb}.

\begin{lemma}
\label{lem:perturb}
Let \(\hat H=\hat H_1\hat H_2\dotsm\hat H_k\) with
\(\hat H_i=I_n-2\hat w_i\hat w_i\herm B\)
satisfying \(\lvert\hat w_i\herm B\hat w_i-1\rvert\leq\epsn\) for \(i=1\),
\(2\), \(\dotsc\), \(k\), where \(B\in\mathbb C^{n\times n}\) is a positive
definite matrix.
Then there exists a Hermitian matrix \(\Delta_k\in\mathbb F^{n\times n}\)
such that
\[
\hat H\herm B\hat H=B^{1/2}(I_n+\Delta_k)B^{1/2}
\]
and
\[
\lVert\Delta_k\rVert_2\leq4k\epsn(1+\epsn)^{4k-2}.
\]
\end{lemma}
\begin{proof}
We prove this lemma by induction.
For \(k=1\), we have
\Eqalign{
\hat H_1\herm B\hat H_1
&=B+4B\hat w_1(\hat w_1\herm B\hat w_1-1)\hat w_1\herm B\\
&=B^{1/2}\bigl(I_n+4\hat z_1(\hat z_1\herm\hat z_1-1)\hat z_1\herm\bigr)
B^{1/2},
}
where \(\hat z_1=B^{1/2}\hat w_1\).
Let \(\Delta_1=4\hat z_1(\hat z_1\herm\hat z_1-1)\hat z_1\herm\).
Then
\[
\lVert\Delta_1\rVert_2
=4\lVert\hat z_1\rVert_2^2\lvert\hat z_1\herm\hat z_1-1\rvert
\leq4\epsn(1+\epsn)^2.
\]
For \(k>1\), we have
\Eqalign{
(\hat H_1\hat H_2\dotsm\hat H_k)\herm B(\hat H_1\hat H_2\dotsm\hat H_k)
&=\hat H_k\herm B\hat H_k+\hat H_k\herm B^{1/2}\Delta_{k-1}B^{1/2}\hat H_k\\
&=\hat H_k\herm B\hat H_k+B^{1/2}(I_n-2\hat z_k\hat z_k\herm)\Delta_{k-1}
(I_n-2\hat z_k\hat z_k\herm)B^{1/2},
}
where \(\hat z_k=B^{1/2}\hat w_k\).
Let
\[
\Delta_k=4\hat z_k(\hat z_k\herm\hat z_k-1)\hat z_k\herm
+(I_n-2\hat z_k\hat z_k\herm)\Delta_{k-1}(I_n-2\hat z_k\hat z_k\herm).
\]
It can be verified that
\[
\lVert I_n-2\hat z_k\hat z_k\herm\rVert_2
=\max\set{1,\lvert1-2\hat z_k\herm\hat z_k\rvert}
\leq1+2\epsn.
\]
Then it follows from the inductive hypothesis that
\Eqalign{
\lVert\Delta_k\rVert_2
&\leq4\lVert\hat z_k\herm(\hat z_k\herm\hat z_k-1)\hat z_k\herm\rVert_2
+\lVert I_n-2\hat z_k\hat z_k\herm\rVert_2^2\lVert\Delta_{k-1}\rVert_2\\
&\leq4\epsn(1+\epsn)^2+(1+2\epsn)^2\cdot4(k-1)\epsn(1+\epsn)^{4k-6}\\
&\leq4\epsn(1+\epsn)^2+4(k-1)\epsn(1+\epsn)^{4k-2}\\
&\leq4k\epsn(1+\epsn)^{4k-2}.
\qedhere
}
\end{proof}

By Lemma~\ref{lem:perturb}, we deduce
\[
\frac{\lVert\hat H\herm B\hat H-B\rVert_2}{\lVert B\rVert_2}
=\frac{\lVert B^{1/2}\Delta_kB^{1/2}\rVert_2}{\lVert B\rVert_2}
\leq\frac{\lVert\Delta_k\rVert_2\lVert B^{1/2}\rVert_2^2}{\lVert B\rVert_2}
\leq4k\epsn(1+\epsn)^{4k-2}
\]
for \(\hat H=\hat H_1\hat H_2\dotsm\hat H_k\).
This result is stated as Theorem~\ref{thm:orth-H} below.

\begin{theorem}
\label{thm:orth-H}
Under the same assumption of Lemma~\ref{lem:perturb}, we have
\[
\frac{\lVert\hat H\herm B\hat H-B\rVert_2}{\lVert B\rVert_2}
\leq4k\epsn(1+\epsn)^{4k-2}.
\]
\end{theorem}

The upper bound in Theorem~\ref{thm:orth-H} is satisfactory in the sense that
it is a small multiple of~\(\epsn\), and, roughly speaking, grows merely
linearly with \(k\).
For instance, if \((4k-2)\epsn\leq1/2\), which can be achieved for reasonably
well-conditioned \(B\)'s, we have
\[
\frac{\lVert\hat H\herm B\hat H-B\rVert_2}{\lVert B\rVert_2}
\leq4k\epsn(1+\epsn)^{4k-2}
\leq\frac{4k\epsn}{1-(4k-2)\epsn}
\leq8k\epsn
=O(k\epsn).
\]
In the case that \(\epsn\) far overestimates the normalization error for the
given set of Householder vectors, \(\epsn\) can be replaced by
\(\max_{1\leq i\leq k}\lvert\hat w_i\herm B\hat w_i-1\rvert\) for a tighter
estimate.
It is worth noting that the error estimate does not depend on the magnitude of
each \(\lVert w_i-\hat w_i\rVert_2\).
Even \(\Span\set{\hat w_i}\)'s are really far away from \(\Span\set{w_i}\)'s,
\(\hat H\) is still close to orthogonal if \(\hat w_i\)'s are properly
normalized.
Only the normalization errors in \(\hat w_i\)'s play a role.

\subsection{Orthogonality of computed orthonormal basis}
Theorem~\ref{thm:orth-H} illustrates the loss of orthogonality for \(\hat H\)
as an approximate unitary operator.
We are also interested in the numerical orthogonality of
\(\hat H\hat U\) in the QR factorization, where~\(\hat U\) is the
(approximate) initial orthonormal basis.
The loss of orthogonality for \(\hat H\hat U\) is provided in
Theorem~\ref{thm:orth-HU}.

\begin{theorem}
\label{thm:orth-HU}
Under the same assumption of Lemma~\ref{lem:perturb}, for any
\(\hat U\in\mathbb F^{n\times k}\) we have
\[
\lVert(\hat H\hat U)\herm B(\hat H\hat U)-I_k\rVert_2
\leq\epso+4k\epsn(1+\epso)(1+\epsn)^{4k-2},
\]
where
\[
\epso=\lVert\hat U\herm B\hat U-I_k\rVert_2.
\]
\begin{proof}
Using Lemma~\ref{lem:perturb}, we have
\Eqalign{
\lVert(\hat H\hat U)\herm B(\hat H\hat U)-I_k\rVert_2
&\leq\lVert\hat U\herm B\hat U-I_k\rVert_2
+\lVert\hat U\herm(\hat H\herm B\hat H-B)\hat U\rVert_2\\
&=\epso+\lVert\hat U\herm B^{1/2}\Delta_k B^{1/2}\hat U\rVert_2\\
&\leq\epso+\lVert B^{1/2}\hat U\rVert_2^2\lVert\Delta_k\rVert_2\\
&=\epso+\lVert\hat U\herm B\hat U\rVert_2\lVert\Delta_k\rVert_2\\
&\leq\epso+4k\epsn(1+\epso)(1+\epsn)^{4k-2}.
\qedhere
}
\end{proof}
\end{theorem}

The error bound in Theorem~\ref{thm:orth-HU} can be roughly understood as
\(O(\epso+k\epsn)\).
We see that there are two sources of errors---the initial orthogonalization
error \(\epso\) contributed by \(\hat U\), and a small multiple of \(\epsn\)
contributed by \(\hat H\).
If \(U\) is constructed by Algorithm~\ref{alg:init}, it can be shown (see,
e.g., \cite[Chapter~14]{Higham2002}) that
\(\epso\leq O(k^3)\kappa_2(\tilde B)\macheps\).
When \(\epso\) is too large, it is recommended to refine \(\hat U\) using,
e.g., extended precision~\cite{YTD2015}.

When \(\hat Q=\hat H\hat U\) needs to be explicitly formed, there is an
additional source of rounding errors characterized in the following
Theorem~\ref{thm:orth-Q}.
Its proof is a bit lengthy, and can be found in the Appendix.
We remark that in the proof of Theorem~\ref{thm:orth-Q}, concrete estimate of
the prefactor in the big-\(O\) notation is not very important since we are
mainly interested in the asymptotic behavior of numerical orthogonality.
Nevertheless, omitting this prefactor does not cause an abuse of the big-\(O\)
notation.

\begin{theorem}
\label{thm:orth-Q}
Let \(\hat U\in\mathbb F^{n\times k}\) whose columns are normalized according
to~\eqref{eq:normalization}, and \(\hat Q=\fl(\hat H\hat U)\).
Under the assumption of Lemma~\ref{lem:perturb} and
\(18k\kappa_2(B)\epsB<1\), we have
\[
\lVert\hat Q\herm B\hat Q-I_k\rVert_2
=\epso+O\bigl(k^2\kappa_2(B)\epsB\bigr),
\]
where \(\epso=\lVert\hat U\herm B\hat U-I_k\rVert_2\leq1\).
\end{theorem}

In summary, our stability analysis suggests that Householder reflections have
nice numerical orthogonality under mild assumptions over \(B\).
In order to attain good orthogonality in practice, it is important to ensure
the orthogonality of the initial orthonormal basis \(\hat U\), and try to
evaluate the \(B\)-inner product as accurately as possible.

It can be seen from the theoretical results that the numerical orthogonality
of Householder reflections may depend on \(\kappa_2(B)\), while does not
depend on \(\kappa_2(X)\).
In fact, the analysis is valid for Householder reflections computed from any
source, not necessarily from orthogonalizing \(X\).
Therefore, \(\kappa_2(X)\) does not play any role here.

\section{Numerical experiments}
\label{sec:experiments}

In the following we present numerical results for our Householder
orthogonalization algorithms.
Algorithm~\ref{alg:init} is used for constructing the initial orthonormal
basis.
Test results for classical and modified Gram--Schmidt processes, with and
without reorthogonalization, are also provided for comparison.
All experiments are performed on a GNU Linux computer running Debian
version~10 (buster) using GNU Octave version 4.4.1 (configured for
\texttt{x86\_64-pc-linux-gnu}) under double-precision arithmetic, where the
machine precision (for real numbers) is
\(\macheps=2^{-53}\approx1.1\times10^{-16}\).

\subsection{Random test matrices}
\label{subsec:random}
We first generate a reasonably well-conditioned Hermitian positive definite
matrix \(B\in\mathbb C^{2000\times2000}\) with \(\kappa_2(B)=10^5\), and test
it with a few different \(X\)'s in \(\mathbb C^{2000\times100}\) by varying
\(\kappa_2(X)\in[10^0,10^{16}]\).
The code snippet for generating random test matrices with prescribed condition
numbers is shown in Figure~\ref{fig:code}.

\begin{figure}[!tb]
\centering
\begin{minipage}{0.9\textwidth}
\begin{lstlisting}
n = 2000; k = 100;
logkappaB = 5;
Q = orth(randn(n) + 1i*randn(n));
U = orth(randn(n, k) + 1i*randn(n, k));
V = orth(randn(k) + 1i*randn(k));
B = Q*diag(logspace(0, -logkappaB, n))*Q'; B = (B + B')/2;
for logkappaX = 0:16,
    X = U*diag(logspace(0, -logkappaX, n))*V;
    ...
end
\end{lstlisting}
\end{minipage}
\caption{Code snippet for generating random test matrices with prescribed
condition numbers.}
\label{fig:code}
\end{figure}

Figures~\ref{fig:err}(a) and~\ref{fig:err}(b) show the losses of orthogonality
and the relative residuals, respectively.
Without reorthogonalization, neither CGS nor MGS retains numerical
orthogonality as \(\kappa_2(X)\) grows.%
\footnote{CGS2 and MGS2, respectively, represent variants of CGS and MGS with
reorthogonalization.}
CGS2, MGS2, and both variants of Householder orthogonalization are numerically
stable, regardless of the magnitude of \(\kappa_2(X)\).
The accuracy of Householder orthogonalization is satisfactory, though being
slightly lower compared to that of CGS2 and MGS2.

We then repeat the same set of tests on another randomly generated
ill-conditioned matrix \(B\) with \(\kappa_2(B)=10^{15}\).
The test results are shown in Figures~\ref{fig:err}(c) and~\ref{fig:err}(d).
All algorithms behave similarly compared to the previous test.
The numerical stability of Householder orthogonalization is again
satisfactory, and is independent of the conditioning of \(X\).

\begin{figure}[!tb]
\centering
\begin{tabular}{cc}
\includegraphics[height=5cm]{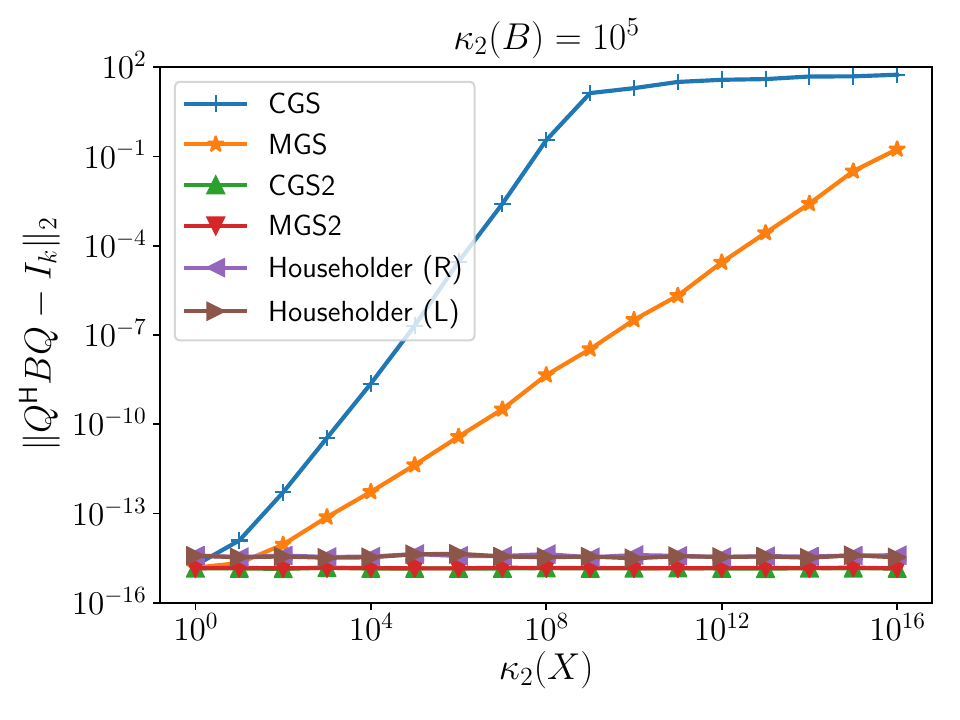} &
\includegraphics[height=5cm]{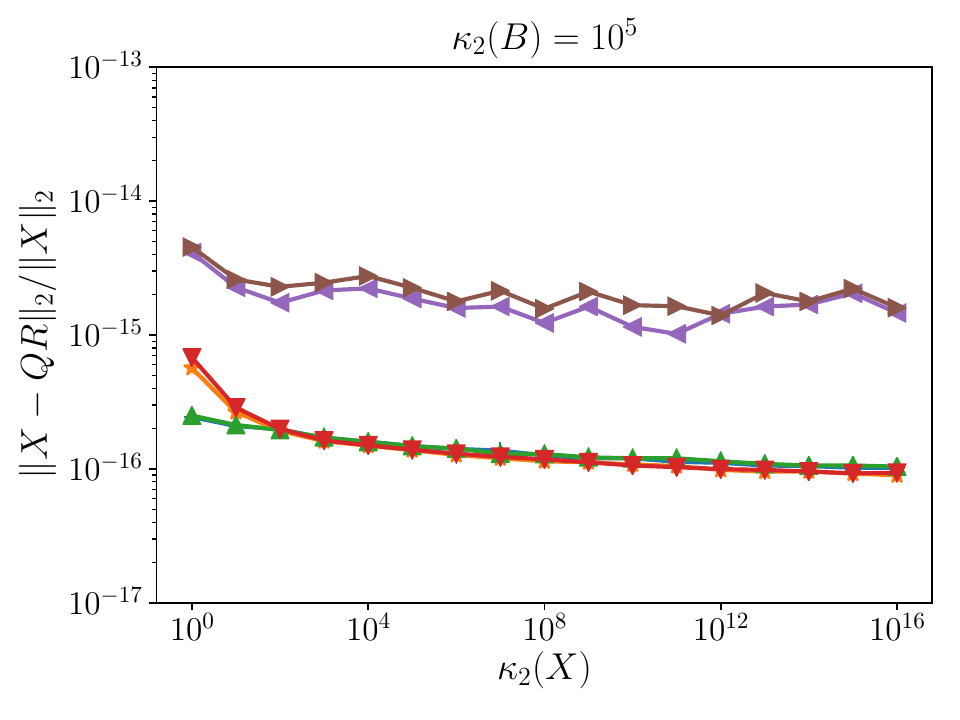} \\
(a) & (b) \\
\includegraphics[height=5cm]{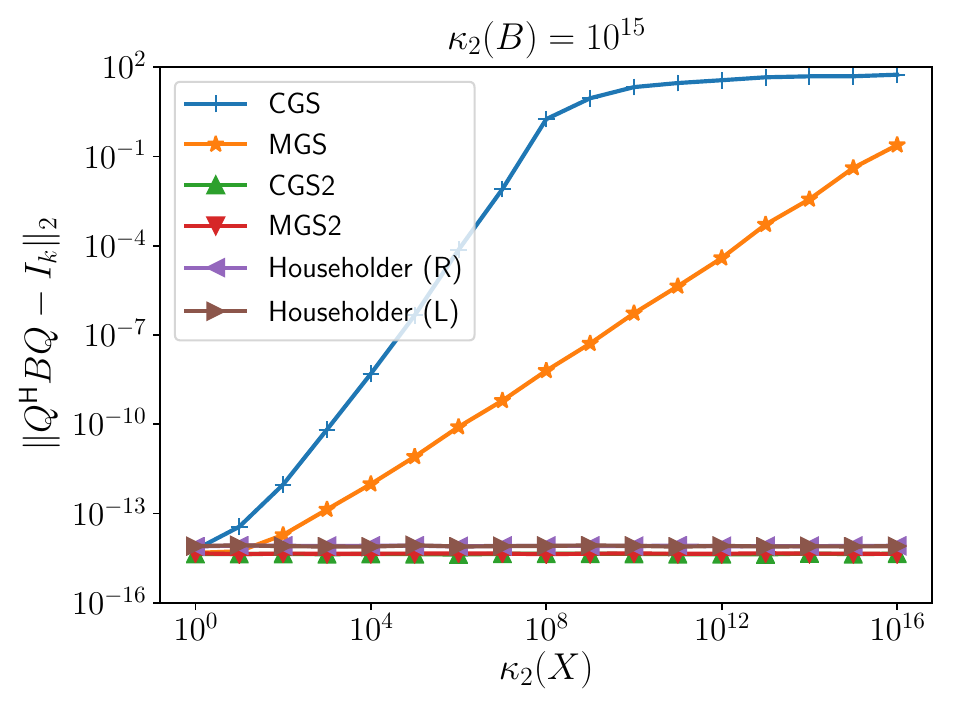} &
\includegraphics[height=5cm]{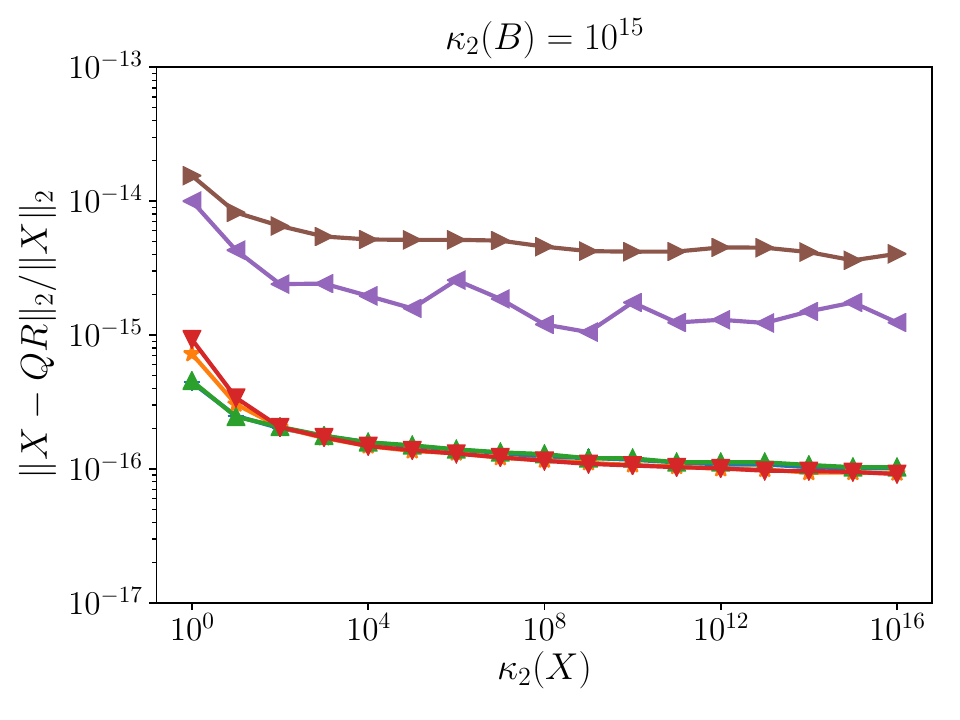} \\
(c) & (d)
\end{tabular}
\caption{Losses of orthogonality (left) and relative residual (right) for
Gram--Schmidt processes and Householder orthogonalization.
`(R)' and `(L)', respectively, stand for right-looking and left-looking
algorithms.}
\label{fig:err}
\end{figure}

As suggested in Section~\ref{sec:algorithm}, in practice a step of
reorthogonalization (Step~\ref{alg-step:reorth-right} in
Algorithm~\ref{alg:right-looking} and Step~\ref{alg-step:reorth-left} in
Algorithm~\ref{alg:left-looking}) is strongly recommended for numerical
stability.
Figure~\ref{fig:reorth} shows the consequence if such a reorthogonalization
step is skipped.
We compare two variants for accumulating the Householder
reflections---\eqref{eq:full} and~\eqref{eq:half}---depending on
whether~\eqref{eq:invariant} is used.
The corresponding variants are labeled as `full' and `half', respectively, in
Figure~\ref{fig:reorth}.
By skipping the reorthogonalization step, the use of~\eqref{eq:half} causes
loss of orthogonality and/or large residual when~\(X\) becomes increasingly
ill-conditioned.
However, when~\eqref{eq:full} is adopted, with the price of almost double
computational cost compared to~\eqref{eq:half}, the numerical orthogonality of
\(Q\) is retained though the residual still increases as \(\kappa_2(X)\)
increases.
This is consistent with the theoretical analysis in
Section~\ref{sec:rounding}---accumulated Householder reflections are
numerically orthogonal as long as the Householder vectors are accurately
normalized.
Interestingly, Figure~\ref{fig:reorth} suggests that sometimes an
accuracy of \(O(\macheps^{1/2})\) can be achieved even for very
ill-conditioned \(X\)'s.
A careful rounding error analysis will be needed in order to explain this
behavior.

\begin{figure}[!tb]
\centering
\begin{tabular}{cc}
\includegraphics[height=5cm]{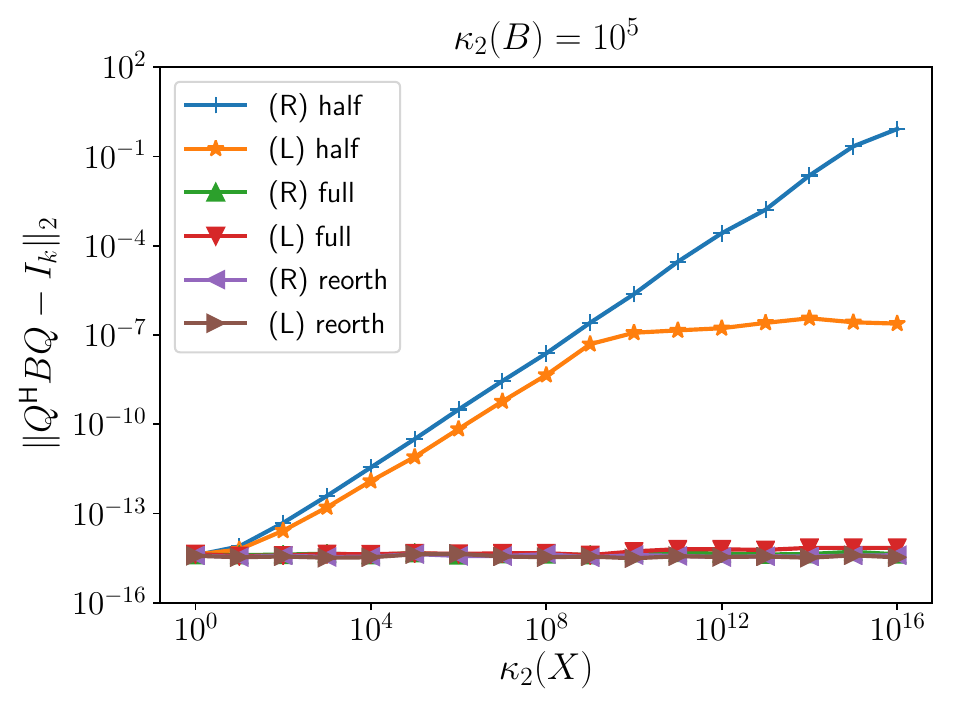} &
\includegraphics[height=5cm]{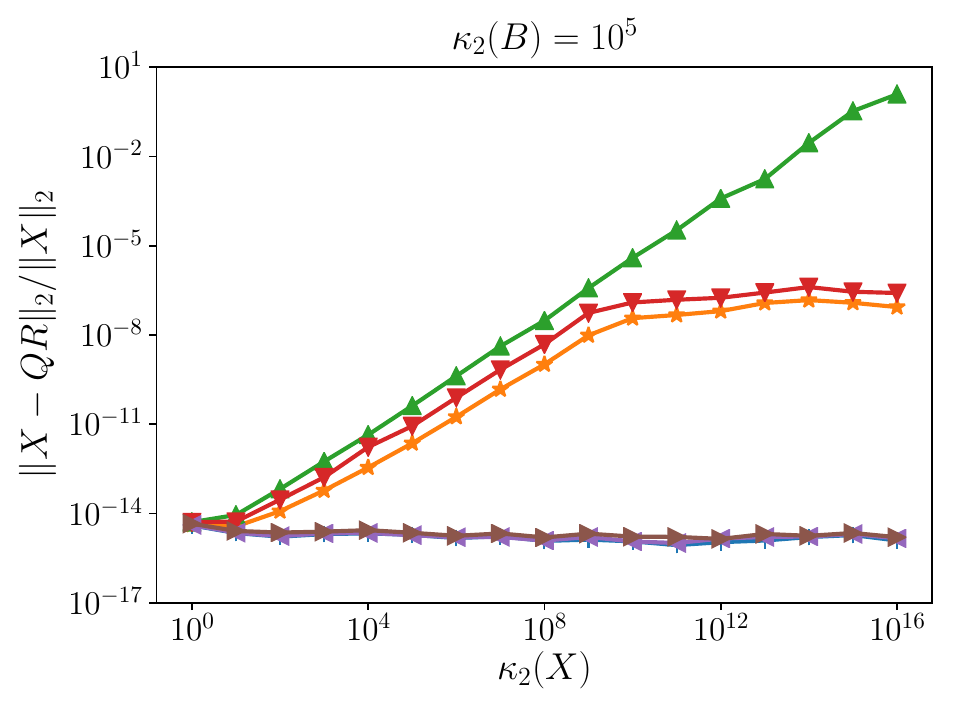} \\
(a) & (b) \\
\includegraphics[height=5cm]{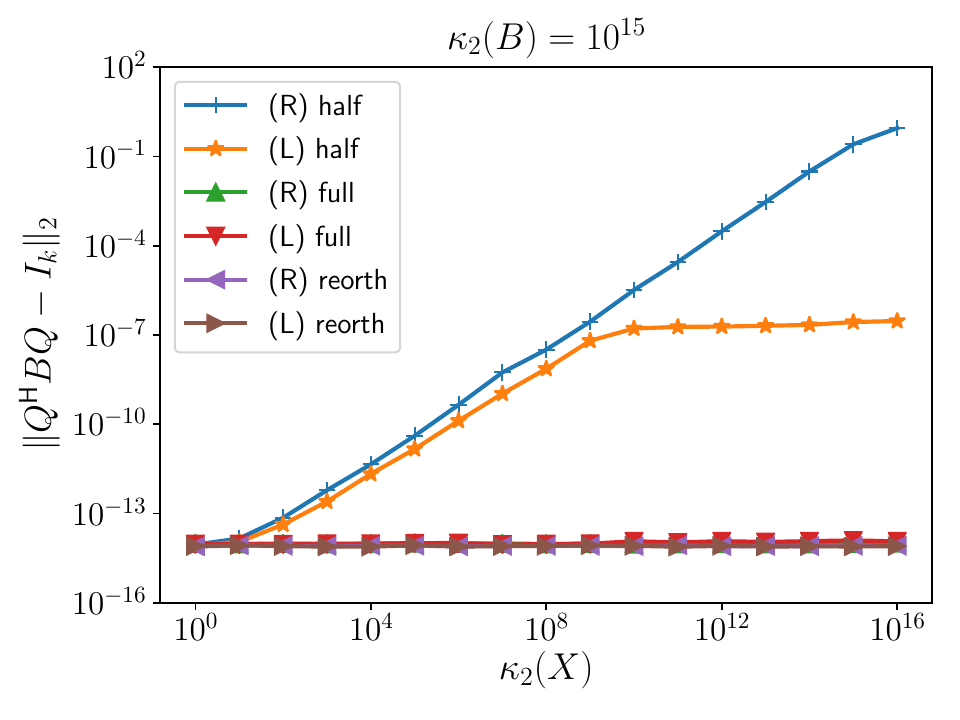} &
\includegraphics[height=5cm]{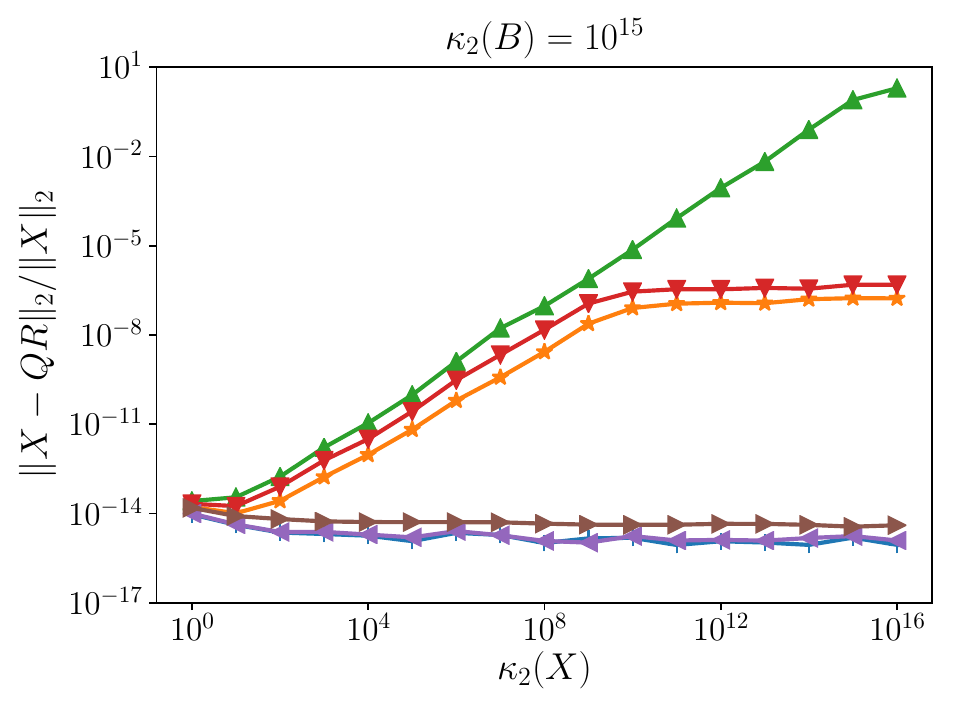} \\
(c) & (c)
\end{tabular}
\caption{Losses of orthogonality (left) and relative residual (right) for
different variants of Householder orthogonalization.
`(R)' and `(L)', respectively, stand for right-looking and left-looking
algorithms.
Labels `full' and `half' are used to denote strategies~\eqref{eq:full}
and~\eqref{eq:half}, respectively, for accumulating the Householder
reflections.
The label `reorth' stands for the `correct' implementation of
Algorithms~\ref{alg:right-looking} and~\ref{alg:left-looking}, both based
on~\eqref{eq:half}.}
\label{fig:reorth}
\end{figure}

\subsection{A rank deficient example}
\label{subsec:rank-deficient}
To illustrate the robustness of Householder orthogonalization, we use an
extremely ill-conditioned test case constructed as follows.
The matrices \(B\in\mathbb C^{2000\times2000}\) and
\(X_0\in\mathbb C^{2000\times10}\) are generated using the technique shown
in Figure~\ref{fig:code} such that \(\kappa_2(B)=\kappa_2(X_0)=10^{20}\).
Numerically, GNU Octave reports \(\texttt{cond}(B)\approx5.9\times10^{19}\) and
\(\texttt{cond}(X_0)\approx2.5\times10^{16}\).
Then we choose \(X\) as
\begin{equation}
\label{eq:rank-deficient}
X=[X_0,0\cdot X_0,X_0]\in\mathbb C^{2000\times30},
\end{equation}
which is rank deficient, and compute the QR factorization of \(X\).
This problem is challenging since both \(B\) and \(X\) are extremely
ill-conditioned.

Table~\ref{tab:err} shows the losses of orthogonality and relative residuals
for this problem.
Since \(X\) is rank deficient and contains zero columns, Gram--Schmidt
processes are implemented in a way that a column \(\hat q_i\) is dropped
when numerically encountering \(\hat q_i\herm B\hat q_i\leq0\) during
normalization.
Then all variants of Gram--Schmidt processes produce \(20\) vectors, while
none of these bases has satisfactory orthogonality.
CGS and CGS2 even yield large residuals.
However, both variants of Householder orthogonalization still work well for
this challenging problem.
Though the matrix~\(B\) is extremely ill-conditioned, its \(30\times30\)
leading principal submatrix \(\tilde B\) is well-conditioned (in fact, GNU
Octave reports \(\kappa_2(\tilde B)\approx7.8\)).
As a result, Algorithm~\ref{alg:init} successfully produces a very good
initial orthonormal basis for this problem.
Then Householder orthogonalization produces an orthonormal basis for a
\(30\)-dimensional subspace, and achieves small factorization error.

\begin{table}
\centering
\caption{Losses of orthogonality and relative residuals for the rank deficient
test problem~\eqref{eq:rank-deficient}.
The column ``\(\rank(Q)\)'' reports the number of vectors in the orthonormal
basis computed by each method.}
\label{tab:err}
\begin{tabular}{cccc}
\hline
Method & \(\rank(Q)\) & \(\lVert Q\herm BQ-I\rVert_2\) &
\(\lVert X-QR\rVert_2/\lVert X\rVert_2\)\vphantom{\(\Big|\)} \\
\hline
CGS             & 20 & \(1.5\times10^1\)     & \(5.7\times10^{-6}\)  \\
MGS             & 20 & \(1.5\times10^0\)     & \(6.7\times10^{-17}\) \\
CGS2            & 20 & \(5.0\times10^0\)     & \(2.5\times10^0\)     \\
MGS2            & 20 & \(1.0\times10^0\)     & \(6.1\times10^{-17}\) \\
Householder (R) & 30 & \(6.5\times10^{-15}\) & \(1.0\times10^{-15}\) \\
Householder (L) & 30 & \(4.5\times10^{-15}\) & \(1.7\times10^{-15}\) \\
\hline
\end{tabular}
\end{table}

\subsection{Examples with numerical difficulties}
\label{subsec:failure}
Previous examples illustrate that in practice the (worst-case) rounding
analysis may be too pessimistic.
Numerical stability can sometimes be expected even for ill-conditioned
matrices.
However, even though rounding errors may be severely overestimated for
concrete examples, the theoretical analysis still suggests potential sources of
numerical difficulties.

In order to construct examples such that the matrix--vector multiplication
\(x\mapsto Bx\) has relatively large rounding errors, we choose \(\Span(X)\)
to be the invariant subspace spanned by eigenvectors corresponding to the five
smallest eigenvalues of \(B\), where \(B\in\mathbb C^{2000\times2000}\) is
randomly generated with prescribed condition numbers as shown in
Figure~\ref{fig:code}.
The eigenvectors are rotated by a random \(5\times5\) unitary matrix to form
the columns of \(X\).
Figures~\ref{fig:kappa}(a) and~\ref{fig:kappa}(b) show the losses of
orthogonality and the relative residuals, respectively, when varying
\(\kappa_2(B)\).
The loss of orthogonality grows linearly with respect to \(\kappa_2(B)\) as
predicted by Theorem~\ref{thm:orth-Q}.
We repeat the same experiment by choosing \(\Span(X)\) to be the invariant
subspace spanned by eigenvectors corresponding to the five largest eigenvalues
of \(B\).
The results are shown in Figures~\ref{fig:kappa}(c) and~\ref{fig:kappa}(d).
The loss of orthogonality in this case remains low as \(\kappa_2(B)\) varies,
likely because the actual normalization error is much lower than \(\epsn\).
These experiments also suggest that the bound on relative residuals may also
depend linearly with \(\kappa_2(B)\).
More insight into the computed residual norms is planned as our future work.

\begin{figure}[!tb]
\centering
\begin{tabular}{cc}
\includegraphics[height=5cm]{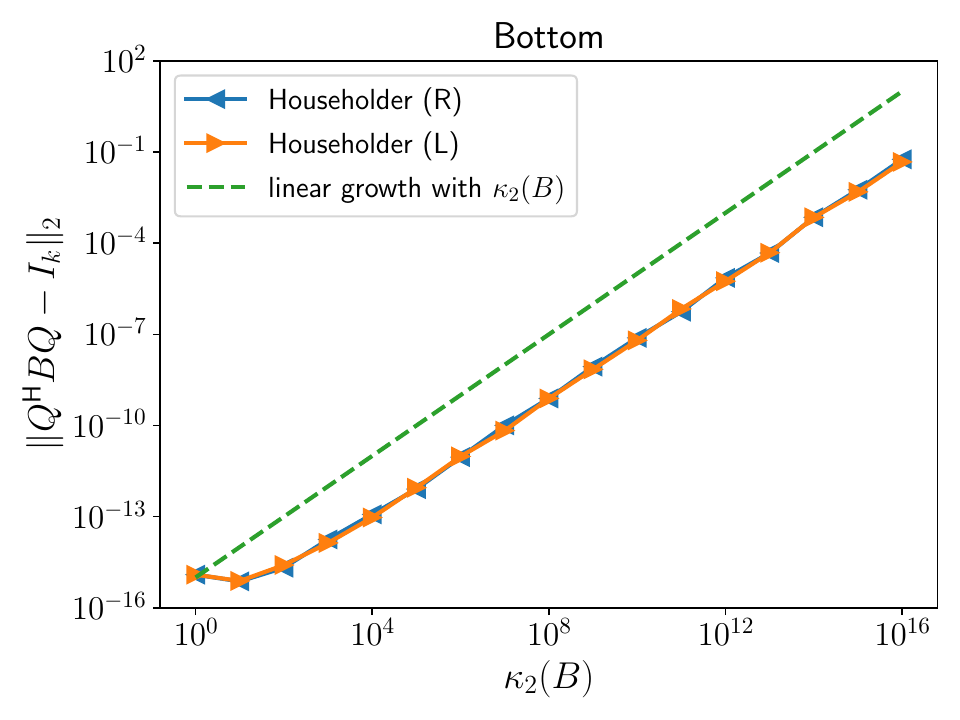} &
\includegraphics[height=5cm]{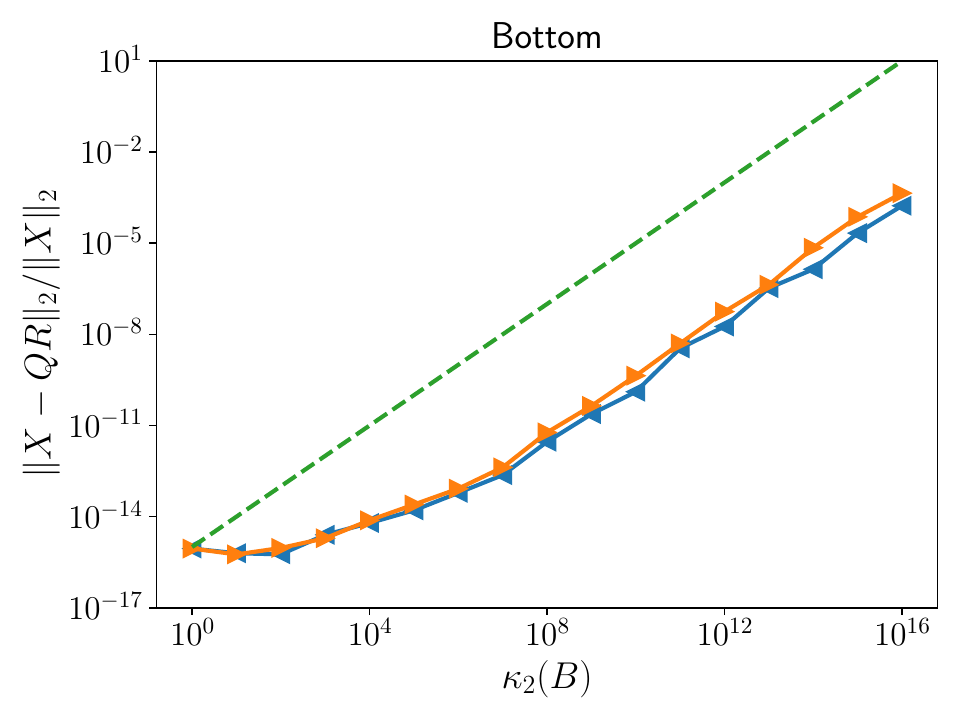} \\
(a) & (b) \\
\includegraphics[height=5cm]{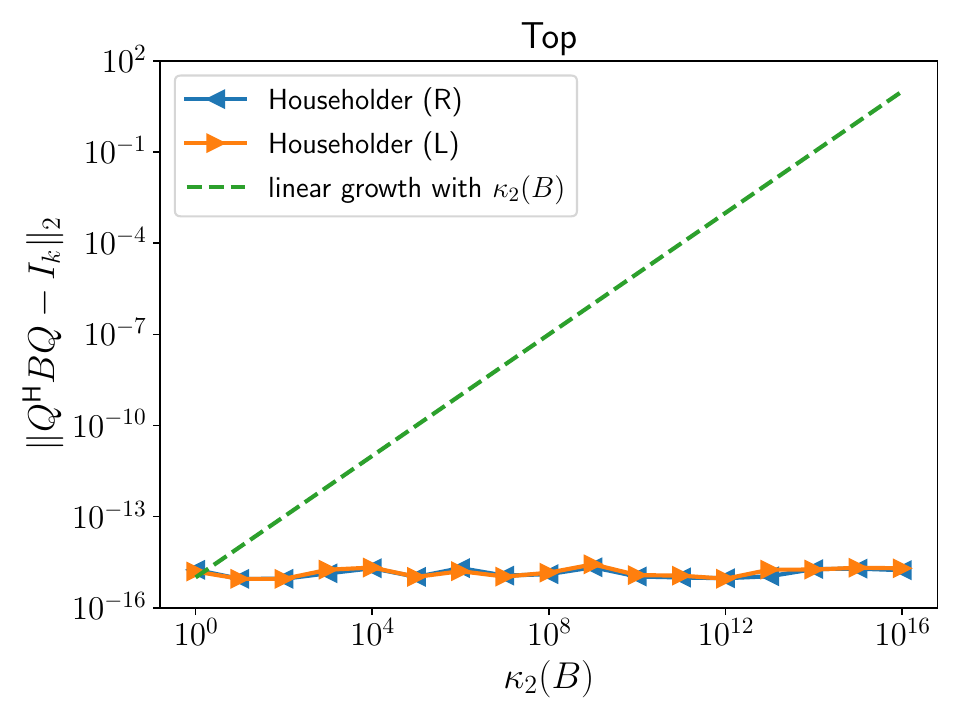} &
\includegraphics[height=5cm]{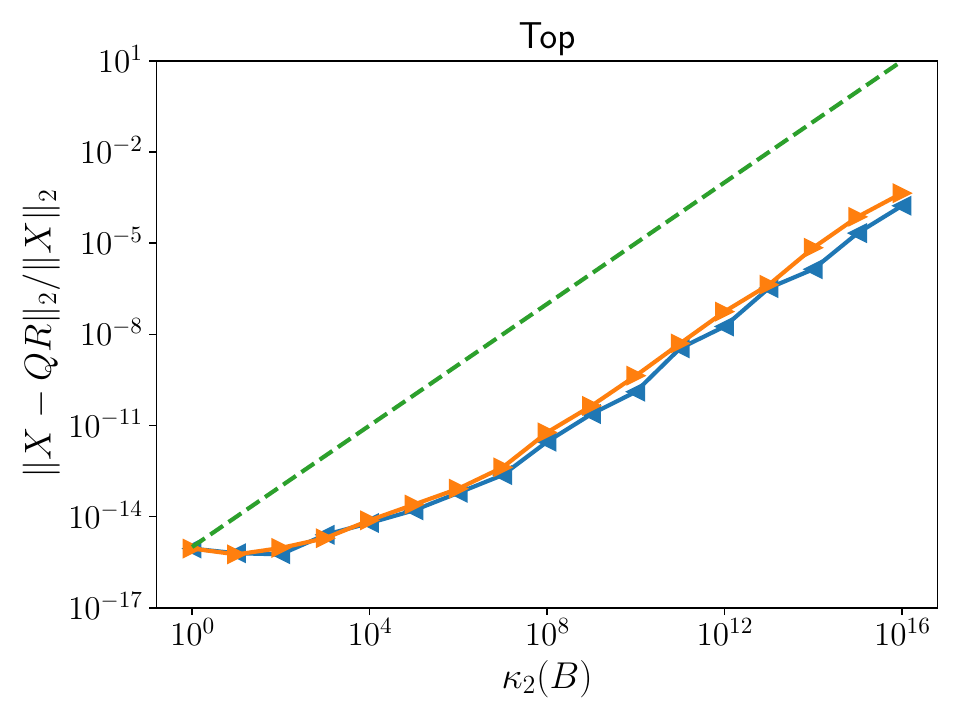} \\
(c) & (d)
\end{tabular}
\caption{Losses of orthogonality (left) and relative residual (right) for
Householder orthogonalization by varying \(\kappa_2(B)\).
`(R)' and `(L)', respectively, stand for right-looking and left-looking
algorithms.
Labels `bottom' and `top', respectively, imply that \(X\) is constructed from
the invariant subspaces corresponding to the bottom and top ends of the
spectrum of \(B\).}
\label{fig:kappa}
\end{figure}

Finally, we show that Householder orthogonalization can sometimes fail, in the
sense that the outputs are not accurate at all.
Let us use the matrix \(B\) from Section~\ref{subsec:rank-deficient},
while \(X\) is chosen to be the eigenvectors corresponding to the five
smallest eigenvalues of \(B\) and normalized in the standard inner product.
Theoretically, \(X\) is perfectly conditioned, and even columns of \(X\) are
already orthogonal in the \(B\)-inner product.
However, since \(B\) is too ill-conditioned, the normalization error bound
\(\epsn\) is expected to be greater than one in this case.
Table~\ref{tab:err2} shows the losses of orthogonality and relative residuals
obtained by different methods.
The numerical orthogonality is lost for all methods due to large normalization
errors.
This agrees with the prediction by the theoretical analysis.
Householder orthogonalization even produces large residuals here.

\begin{table}[!tb]
\centering
\caption{Losses of orthogonality and relative residuals for \(X\)
consisting eigenvectors corresponding to the smallest eigenvalues of an
ill-conditioned \(B\).}
\label{tab:err2}
\begin{tabular}{cccc}
\hline
Method & \(\lVert Q\herm BQ-I\rVert_2\) &
\(\lVert X-QR\rVert_2/\lVert X\rVert_2\)\vphantom{\(\Big|\)} \\
\hline
CGS             & \(5.5\times10^0\)     & \(1.5\times10^{-15}\)  \\
MGS             & \(6.3\times10^0\)     & \(2.0\times10^{-15}\) \\
CGS2            & \(7.8\times10^0\)     & \(4.0\times10^{-15}\)     \\
MGS2            & \(1.0\times10^0\)     & \(2.5\times10^{-15}\) \\
Householder (R) & \(2.6\times10^2\)     & \(1.1\times10^2\) \\
Householder (L) & \(2.4\times10^4\)     & \(2.4\times10^4\) \\
\hline
\end{tabular}
\end{table}

\section{Concluding remarks}
\label{sec:conclusions}
In this paper we have discussed how to compute a tall-skinny QR factorization
\(X=QR\) with the \(B\)-inner product using Householder reflections.
Algorithmic variants as well as strategies for constructing an initial
orthonormal basis are discussed.
Theoretical analysis and numerical experiments demonstrate that Householder
orthogonalization is numerically stable under mild assumptions over \(B\),
even if \(X\) is rank deficient.

The rounding error analysis in this paper focuses on the numerical
orthogonality.
Rounding errors on the residual norm are only reported without theoretical
analysis in this work.
A more complete analysis regarding the Householder orthogonalization process,
especially for difference between right-looking and left-looking variants,
with and without the reorthogonalization step, will be helpful for better
understanding the proposed methods.
The robustness of Householder orthogonalization provides opportunities for new
algorithmic variants for solving weighted least squares problems and
generalized eigenvalue problems.
Developments in these directions are planned as our future work.

\section*{Acknowledgments}
We are grateful to Zhaojun Bai, Weiguo Gao, Daniel Kressner, Yuqi Liu, and
Chao Yang for helpful discussions.
We also thank anonymous referees for insightful comments that largely improves
the paper compared to its initial form.

\section*{Appendix---Proof of Theorem~\ref{thm:orth-Q}}
In order to prove Theorem~\ref{thm:orth-Q}, we first establish
Lemma~\ref{lem:Householder-1} for the rounding error of applying a Householder
reflection to a vector in finite precision arithmetic.
We assume that \((I_n-2ww\herm B)y\) is computed through \(\eta=w\herm By\)
and then \(y-2w\eta\).

\begin{lemma}
\label{lem:Householder-1}
Let \(B\in\mathbb C^{n\times n}\) be positive definite.
Suppose that \(z=(I_n-2\hat w\hat w\herm B)y\) with
\(\lvert\hat w\herm B\hat w-1\rvert\leq\epsn<1\) and
\(w=\hat w/\lVert\hat w\rVert_B\).
Under the assumption~\eqref{eq:innerprod}, there exists
\(\Delta H\in\mathbb F^{n\times n}\) such that
\[
\fl(z)=(I_n-2ww\herm B+\Delta H)y
\]
and
\[
\lVert\Delta H\rVert_B
=\lVert B^{1/2}\Delta HB^{-1/2}\rVert_2
\leq\delta,
\]
where
\begin{align}
\label{eq:delta}
\delta={}&2\epsn+(1+2\epsn)\kappa_2(B)^{1/2}\macheps\\
&+2\bigl(1+\kappa_2(B)^{1/2}\macheps\bigr)(1+\epsn)
\bigl(\kappa_2(B)\epsB+(1+\kappa_2(B)\epsB)\kappa_2(B)^{1/2}\macheps\bigr).
\nonumber
\end{align}
\end{lemma}
\begin{proof}
For simplicity, \emph{within this proof} we use \(\hat\cdot=\fl(\cdot)\)
interchangeably without further clarification.
Let \(\eta=\hat w\herm By\).
We have \(\hat\eta=\hat w\herm(B+\Delta B)y\) with
\(\lVert\Delta B\rVert_2 \leq\epsB\lVert B\rVert_2\).
Then
\[
\lvert\hat\eta-\eta\rvert
\leq\lVert B^{1/2}\hat w\rVert_2\lVert B^{-1/2}\Delta BB^{-1/2}\rVert_2
\lVert B^{1/2}y\rVert_2
\leq(1+\epsn)^{1/2}\kappa_2(B)\epsB\lVert y\rVert_B,
\]
and
\[
\lvert\hat\eta\rvert
\leq\lvert\eta\rvert+\lvert\hat\eta-\eta\rvert
\leq\lVert\hat w\rVert_B\lVert y\rVert_B+\lvert\hat\eta-\eta\rvert
\leq(1+\epsn)^{1/2}\bigl(1+\kappa_2(B)\epsB\bigr)\lVert y\rVert_B.
\]
Let \(v=2\hat w\eta\), \(\hat v=\fl(v)=\fl(2\hat w\hat\eta)\).
Then
\begin{align*}
\lVert v-\hat v\rVert_B
&\leq\lVert v-2\hat w\hat\eta\rVert_B+\lVert2\hat w\hat\eta-\hat v\rVert_B\\
&\leq2\lVert\hat w\rVert_B\lvert\eta-\hat\eta\rvert
+\lVert B^{1/2}\rVert_2\lVert2\hat w\hat\eta-\hat v\rVert_2\\
&\leq2\lVert\hat w\rVert_B\lvert\eta-\hat\eta\rvert
+\macheps\lVert B^{1/2}\rVert_2\lVert2\hat w\hat\eta\rVert_2\\
&\leq2\lVert\hat w\rVert_B\lvert\eta-\hat\eta\rvert
+2\lvert\hat\eta\rvert\kappa_2(B)^{1/2}\macheps\lVert\hat w\rVert_B\\
&\leq2(1+\epsn)\bigl(\kappa_2(B)\epsB+(1+\kappa_2(B)\epsB)
\kappa_2(B)^{1/2}\macheps\bigr)\lVert y\rVert_B.
\end{align*}
Note that
\[
\lVert(I_n-2\hat w\hat w\herm B)-(I_n-2ww\herm B)\rVert_B
=\lvert\hat w\herm B\hat w-1\rvert\cdot\lVert2ww\herm B\rVert_B
\leq2\epsn.
\]
We obtain
\[
\lVert z\rVert_B\leq
\lVert(I_n-2ww\herm B)y\rVert_B
+\lVert(I_n-2\hat w\hat w\herm B)y-(I_n-2ww\herm B)y\rVert_B
\leq(1+2\epsn)\lVert y\rVert_B.
\]
The forward error for evaluating \(z\) is given by
\[
r=\hat z-z=\fl(y-\hat v)-(y-\hat v)+(v-\hat v).
\]
Hence
\begin{align*}
\lVert r\rVert_B
&\leq\lVert\fl(y-\hat v)-(y-\hat v)\rVert_B+\lVert v-\hat v\rVert_B\\
&\leq\lVert B^{1/2}\rVert_2\lVert\fl(y-\hat v)-(y-\hat v)\rVert_2
+\lVert v-\hat v\rVert_B\\
&\leq\macheps\lVert B^{1/2}\rVert_2\lVert y-\hat v\rVert_2
+\lVert v-\hat v\rVert_B\\
&\leq\kappa_2(B)^{1/2}\macheps
\bigl(\lVert z\rVert_B+\lVert v-\hat v\rVert_B\bigr)
+\lVert v-\hat v\rVert_B\\
&\leq(1+2\epsn)\kappa_2(B)^{1/2}\macheps\lVert y\rVert_B
+\bigl(1+\kappa_2(B)^{1/2}\macheps\bigr)\lVert v-\hat v\rVert_B.
\end{align*}

Without loss of generality, we assume that \(y\neq0\).
Let \(\Delta\hat H=ry\herm B/(y\herm By)\).
It is easy to verify that
\[
(I_n-2\hat w\hat w\herm B+\Delta\hat H)y=z+r=\hat z
\]
and
\begin{multline*}
\lVert\Delta\hat H\rVert_B
=\frac{\lVert r\rVert_B}{\lVert y\rVert_B}\\
\leq(1+2\epsn)\kappa_2(B)^{1/2}\macheps
+2\bigl(1+\kappa_2(B)^{1/2}\macheps\bigr)(1+\epsn)
\bigl(\kappa_2(B)\epsB+(1+\kappa_2(B)\epsB)\kappa_2(B)^{1/2}\macheps\bigr).
\end{multline*}
By setting
\[
\Delta H=\Delta\hat H+(I_n-2\hat w\hat w\herm B)-(I_n-2ww\herm B),
\]
we arrive at the conclusion.
\end{proof}

The following Lemma~\ref{lem:Householder-k} provides the tool to analyze the
error for evaluating a sequence of Householder reflections.

\begin{lemma}
\label{lem:Householder-k}
Let \(H_i=I_n-2w_iw_i\herm B\) with \(\lVert w_i\rVert_B=1\) for \(i=1\),
\(2\), \(\dotsc\), \(k\).
Then
\[
\lVert(H_1+\Delta H_1)(H_2+\Delta H_2)\dotsm(H_k+\Delta H_k)y
-H_1H_2\dotsm H_ky\rVert_B
\leq\bigl((1+\delta)^k-1\bigr)\lVert y\rVert_B,
\]
for \(\delta\geq\max_{1\leq i\leq k}\lVert\Delta H_i\rVert_B\).
\end{lemma}
\begin{proof}
Notice that \(\tilde H_i=B^{1/2}H_iB^{-1/2}\) is a unitary matrix for \(i=1\),
\(2\), \(\dotsc\), \(k\).
It follows from~\cite[Lemma~3.7]{Higham2002} that
\begin{align*}
&\lVert(H_1+\Delta H_1)\dotsm(H_k+\Delta H_k)-H_1\dotsm H_k\rVert_B\\
={}&\lVert(\tilde H_1+B^{1/2}\Delta H_1B^{-1/2})
\dotsm(\tilde H_k+B^{1/2}\Delta H_kB^{-1/2})
-\tilde H_1\dotsm\tilde H_k\rVert_2\\
\leq{}&\biggl(\prod_{i=1}^k\bigl(1+\lVert B^{1/2}\Delta H_iB^{-1/2}\rVert_2
\bigr)-1\biggr)\prod_{i=1}^k\lVert\tilde H_i\rVert_2\\
\leq{}&(1+\delta)^k-1.
\end{align*}
Therefore,
\begin{align*}
&\lVert(H_1+\Delta H_1)(H_2+\Delta H_2)\dotsm(H_k+\Delta H_k)y
-H_1H_2\dotsm H_ky\rVert_B\\
\leq{}&\lVert(H_1+\Delta H_1)(H_2+\Delta H_2)\dotsm(H_k+\Delta H_k)
-H_1H_2\dotsm H_k\rVert_B\lVert y\rVert_B\\
\leq{}&\bigl((1+\delta)^k-1\bigr)\lVert y\rVert_B.
\qedhere
\end{align*}
\end{proof}

Now we are ready to proof Theorem~\ref{thm:orth-Q}.
Let us define \(\delta\) as in~\eqref{eq:delta}.
It follows from Lemmas~\ref{lem:Householder-1} and~\ref{lem:Householder-k}
that
\[
\lVert B^{1/2}\hat q_i-B^{1/2}(\hat H\hat u_i)\rVert_2
=\lVert\hat q_i-\hat H\hat u_i\rVert_B
\leq\bigl((1+\delta)^k-1\bigr)\lVert\hat u_i\rVert_B
\leq(1+\epsn)^{1/2}\bigl((1+\delta)^k-1\bigr),
\]
where \(\hat q_i\) and \(\hat u_i\) denote the \(i\)th columns of \(\hat Q\)
and \(\hat U\), respectively.
Therefore,
\[
\lVert B^{1/2}\hat Q-B^{1/2}(\hat H\hat U)\rVert_2
\leq\lVert B^{1/2}\hat Q-B^{1/2}(\hat H\hat U)\rVert_{\fro}
\leq k^{1/2}(1+\epsn)^{1/2}\bigl((1+\delta)^k-1\bigr).
\]
It follows from Theorem~\ref{thm:orth-HU} that
\[
\lVert B^{1/2}(\hat H\hat U)\rVert_2
\leq\lVert(\hat H\hat U)\herm B(\hat H\hat U)\rVert_2^{1/2}
\leq(1+\epso)^{1/2}\bigl(1+4k\epsn(1+\epsn)^{4k-2}\bigr)^{1/2}.
\]
Using the inequality
\begin{align*}
\lVert X\herm X-Y\herm Y\rVert_2
&=\lVert(X-Y)\herm Y+Y\herm(X-Y)+(X-Y)\herm(X-Y)\rVert_2\\
&\leq2\lVert(X-Y)\herm Y\rVert_2+\lVert (X-Y)\herm(X-Y)\rVert_2\\
&\leq2\lVert X-Y\rVert_2\lVert Y\rVert_2+\lVert X-Y\rVert_2^2,
\end{align*}
we obtain
\begin{align*}
\lVert\hat Q\herm B\hat Q-I_k\rVert_2
&\leq\lVert\hat Q\herm B\hat Q-(\hat H\hat U)\herm B(\hat H\hat U)\rVert_2
+\lVert(\hat H\hat U)\herm B(\hat H\hat U)-I_k\rVert_2\\
&\leq2\lVert B^{1/2}(\hat H\hat U)\rVert_2
\lVert B^{1/2}(\hat Q-\hat H\hat U)\rVert_2
+\lVert B^{1/2}(\hat Q-\hat H\hat U)\rVert_2^2\\
&\quad+\lVert(\hat H\hat U)\herm B(\hat H\hat U)-I_k\rVert_2.
\end{align*}
Every term in the upper bound has already been analyzed.

In order to derive a neat estimate without high order terms such as
\(O(\epsB^2)\), we make use of the inequality
\[
(1+a)^b\leq\frac{1}{1-ab},
\qquad (\text{if \(a>0\), \(b>0\), \(ab<1\)}).
\]
Under the assumptions that \(\epso\leq1\) and \(18k\kappa_2(B)\epsB<1\), we
have
\begin{align*}
\delta\leq\frac{568}{153}\kappa_2(B)^{1/2}\macheps
+\frac{38}{17}\kappa_2(B)\epsB+2\epsn
\leq9\kappa_2(B)\epsB,\\
(1+\delta)^k-1\leq\frac{k\delta}{1-k\delta}
\leq2k\delta
\leq18k\kappa_2(B)\epsB
<1,
\end{align*}
and
\[
(1+\epsn)^{4k-2}\leq\frac{1}{1-4k\epsn}\leq\frac{17}{13}.
\]
Finally, we obtain
\begin{align*}
\lVert\hat Q\herm B\hat Q-I_k\rVert_2
\leq{}&2(1+\epso)^{1/2}\bigl(1+4k\epsn(1+\epsn)^{4k-2}\bigr)^{1/2}
\cdot k^{1/2}(1+\epsn)^{1/2}\bigl((1+\delta)^k-1\bigr)\\
&+k(1+\epsn)\bigl((1+\delta)^k-1\bigr)^2
+\epso+4k\epsn(1+\epso)(1+\epsn)^{4k-2}\\
\leq{}&\epso+\frac{136}{13}k\epsn+\frac{216}{\sqrt{13}}k^{3/2}\kappa_2(B)\epsB
+\frac{324}{17}k^2\kappa_2(B)\epsB\\
={}&\epso+O\bigl(k^2\kappa_2(B)\epsB\bigr).
\end{align*}
This completes the proof of Theorem~\ref{thm:orth-Q}.

\end{document}